\documentclass[11pt,a4paper]{article}
\usepackage[utf8]{inputenc}
\usepackage[T1]{fontenc}
\usepackage[french]{babel}
\usepackage{amsmath,amssymb,amsthm}
\usepackage{geometry}
\usepackage{xcolor}
\usepackage{hyperref}
\usepackage{enumitem}

\hypersetup{
    colorlinks=true,
    linkcolor=blue,
    citecolor=red,
    urlcolor=black
}

\newcommand{\owedge}{\mathbin{\bigcirc\hspace{-9.5pt}\wedge\;}}

\newtheorem{thm}{Theorem}[section]
\newtheorem{prop}[thm]{Proposition}
\newtheorem{lem}[thm]{Lemma}
\newtheorem{defi}[thm]{Definition}
\newtheorem{rem}[thm]{Remark}
\newtheorem{cor}[thm]{Corollary}
\newtheorem{ex}[thm]{Example}

\newcommand{\doi}[1]{%
    \href{https://doi.org/#1}{\texttt{#1}}%
}

\title{
\textbf{Ricci solitons, Yamabe's almost-$\theta$ solitons, and symmetries of finite-order tensors of a vector field on Riemannian varieties with anisotropic curvature of rank one}
}

\author{
  \textbf{Abdou Bousso}\thanks{Département de Mathématiques et Informatique,
Université Cheikh Anta Diop de Dakar, Sénégal.\\
E-mail : \href{mailto:abdoukskbousso@gmail.com}{\texttt{abdoukskbousso@gmail.com}}}
  \quad and \quad 
  \textbf{Ameth Ndiaye}\thanks{
Département de Mathématiques, FASTEF,
Université Cheikh Anta Diop de Dakar, Sénégal.\\ E-mail : \href{mailto:ameth1.ndiaye@ucad.edu.sn}{\texttt{ameth1.ndiaye@ucad.edu.sn}}}
}
\date{}
\addto\captionsfrench{}
\begin{document}

\maketitle
\renewcommand{\abstractname}{Abstract}
\begin{abstract}
We study the iterated action of the Lie derivative on the curvature tensor of a Riemannian manifold of quasi-constant curvature, where the Riemann tensor is expressed via the Kulkarni–Nomizu product as $R = \lambda (\xi^\flat\otimes\xi^\flat)\owedge g$. First, we examine the conditions under which such a manifold is a Ricci soliton and demonstrate that the Ricci soliton property implies the almost-$\theta$ Yamabe soliton property. We also show that if the manifold is a Ricci soliton whose associated vector field is a symmetry of the Ricci tensor of a fixed order $k$, the problem reduces to solving a partial differential equation of order $k+1$ relating its flow to its $k$-th order Lie derivative with respect to the metric. The paper concludes by demonstrating that the geometric problem whether establishing the relation $\mathcal{L}_X^kR=R$ for a fixed integer $k$, determining the minimum order of $R$ with respect to $X$ for $X$ to be a Lie symmetry of the curvature tensor of order $k$, or satisfying the relation $\mathcal{L}_X^{k+1}R=f\mathcal{L}_X^kR$ (where $f$ is a continuous function on the manifold) is equivalent to the differential problem induced by the operator $D_X = X + 6\varphi + 2a$, under the assumption that $X$ is a conformal vector field ($\mathcal{L}_Xg=2\varphi g$) whose infinitesimal flow preserves the line distribution $\mathcal{D}=\operatorname{Span}\{\xi\}$ (with $[X,\xi]=a\xi$ and $a\in\mathbb{R}$).
\medskip

\medskip


\end{abstract}

\section{Introduction}

A deep understanding of Riemannian and semi-Riemannian geometry relies fundamentally on a detailed analysis of the Riemann tensor and its intrinsic symmetries \cite{doCarmo, Lee, ONeill}. Historically, the investigation of manifolds subject to specific curvature constraints—whether involving non-positive curvature or models with quasi-constant curvature—has revealed particularly rich topological and geometric properties \cite{BallmannGromovSchroeder, Eberlein, Heintze, Karcher}. Building on this foundation, the global characterization of these spaces has been significantly deepened through the study of Einstein manifolds and their metric deformations \cite{Besse, Thurston1997}.

At the heart of this work lie Riemannian manifolds endowed with a singular curvature structure, which we term "rank-one anisotropic." More precisely, we assume that the curvature tensor is expressed via the Kulkarni-Nomizu product in the following form:
\begin{equation}
    R = \lambda (\xi^\flat\otimes\xi^\flat)\owedge g, \label{eq1}
\end{equation}
where $\lambda$ denotes a differentiable function and $\xi$ a smooth vector field on the manifold. This formalism widely utilized in the analysis of quasi-Einstein manifolds and the study of Riemannian submersions \cite{Derdzinski} provides a privileged algebraic and geometric framework for assessing the influence of the one-dimensional distribution $\mathcal{D} = \operatorname{Span}\{\xi\}$ on the ambient geometry.

Concurrently, the theory of geometric flows, driven by seminal work on Ricci and Yamabe flows \cite{Chow1992, Hamilton}, has highlighted the pivotal role of solitons. These self-similar metrics are currently experiencing significant growth and appear in numerous variants (such as Ricci-Bourguignon solitons, Ricci-Yamabe solitons, and $(h,\eta)$-almost solitons). Indeed, recent research has established classification and existence theorems for these structures on hyperbolic spaces, Lie groups, and three-dimensional homogeneous spaces \cite{BoussoNdiayeJDSGT2025, BoussoNdiaye2026Hn, BoussoNdiayeH2R2025, BoussoNdiayeSol3, DiopBoussoNdiayeMandal2026}.

The central aim of this article is to explore the profound interplay between the anisotropic curvature condition of equation \eqref{eq1} and the existence of these generalized solitonic configurations. Drawing on the powerful tools of tensor calculus and the Lie derivative \cite{Yano}, we investigate the metric symmetries generated by a conformal vector field $X$ that preserves the characteristic distribution $\mathcal{D}$.

In this paper, we demonstrate how higher-order symmetries translate into the solution of complex systems of partial differential equations. Finally, we highlight the stationary and asymptotic behavior of iterations of the Lie derivative of the curvature tensor, thereby offering new insights into the geometric rigidity of these spaces.
\section{Preliminaries}
\addcontentsline{toc}{section}{Preliminaries}
Let $(M,g)$ be a Riemannian manifold of dimension $n\geq2$.
We denote by $\nabla$ the Levi-Civita connection of $g$.

We adopt the convention
\begin{equation}\label{Rr}
    R(X,Y)Z=\nabla_X\nabla_YZ
-\nabla_Y\nabla_XZ
-\nabla_{[X,Y]}Z.
\end{equation}
The $(0,4)$-curvature tensor is
$R(X,Y,Z,W)=g(R(X,Y)Z,W).$

Let $h$ and $k$ be two symmetric tensors of type $(0,2)$.

\begin{defi}
The Kulkarni--Nomizu product of $h$ and $k$ is the tensor
$h\owedge k$ defined by
\begin{align}\label{produit}
(h\owedge k)(X,Y,Z,W)
=&h(X,W)k(Y,Z)+h(Y,Z)k(X,W)
\nonumber\\
&-h(X,Z)k(Y,W)-h(Y,W)k(X,Z).
\end{align}
In coordinates,

\begin{equation}
\label{eq:KN}
(h\owedge k)_{ijkl}
=h_{il}k_{jk}+h_{jk}k_{il}-h_{ik}k_{jl}-h_{jl}k_{ik}.
\end{equation}
\end{defi}

When $h=k=g$, we obtain
$$(g\owedge g)(X,Y,Z,W)=2\Big(g(X,W)g(Y,Z)-g(X,Z)g(Y,W)\Big).$$
\begin{defi}
    Let $X\in\mathfrak{X}(M)$. For each point $p\in M$, an integral curve
of $X$ passing through $p$ is a curve $\gamma_p:I_p\longrightarrow M$
satisfying
\begin{equation}
    \gamma_p'(t)
    =
    X_{\gamma_p(t)} \quad\text{and}\quad  \gamma_p(0)=p.
    \label{eq:trajectoire}
\end{equation}

The local flow is defined by $\Phi_t(p)=\gamma_p(t).$

It satisfies
\begin{equation}
    \Phi_0=\operatorname{Id} \quad\text{and}\quad  \frac{d}{dt}\Phi_t(p)
    =  X_{\Phi_t(p)}.
    \label{eq:flow-zero}
\end{equation}
\end{defi}
\begin{defi}
    Let $X\in\mathfrak{X}(M)$. The Lie derivative of a tensor $\mathcal{W}$ with respect to $X$ is defined
by
\begin{equation}
    \mathcal{L}_X\mathcal{W}
    =
    \left.
    \frac{d}{dt}
    \right|_{t=0}
    \varphi_t^*\mathcal{W},
    \label{eq:def-lie}
\end{equation}
where $\varphi_t$ denotes the local flow of $X$.

We set
\begin{equation}
    \mathcal{L}_X^0\mathcal{W}=\mathcal{W}
    \label{eq:lie-zero}
\end{equation}
and, for $k\geq1$,
\begin{equation}
    \mathcal{L}_X^k\mathcal{W}= \mathcal{L}_X\left( \mathcal{L}_X^{k-1}\mathcal{W} \right).\label{eq:lie-k}
\end{equation}

\end{defi}

\begin{defi}
Let $X\in\mathfrak X(M)$ and $k\geq1$.
We say that \begin{itemize}[label=$\bullet$]
 \item $X$ is a \textbf{metric symmetry of order $k$} if $\mathcal{L}_X^kg=0;$
    \item $X$ is a \textbf{curvature symmetry of order $k$} if $\mathcal{L}_X^kR=0;$
 \item $X$ is a \textbf{Ricci tensor symmetry of order $k$} if $\mathcal{L}_X^k\operatorname{Ric}=0.$
\end{itemize}
\end{defi}
\begin{defi}
Let $X\in\mathfrak X(M)$ and $k\geq1$.
We say that \begin{itemize}[label=$\star$]
 \item $X$ is a \textbf{Lie symmetry of the metric of order $k$} if $\mathcal{L}_X^kg=0\quad \text{and}\quad  \mathcal{L}_X^{k-1}g\neq0;$
    \item $X$ is a \textbf{Lie symmetry of the curvature of order $k$} if $\mathcal{L}_X^kR=0\quad \text{and} \quad\mathcal{L}_X^{k-1}R\neq0;$
 \item $X$ is a \textbf{Lie symmetry of the Ricci tensor of order $k$} if $\mathcal{L}_X^k\operatorname{Ric}=0\quad \text{and} \quad\mathcal{L}_X^{k-1}\operatorname{Ric}\neq0.$
\end{itemize}
\end{defi}
\begin{defi}
Let $\mathcal{W}$ be a tensor of type $(r,s)$. When the set $\{k\geq1:\mathcal{L}_X^k\mathcal{W}=0\}$ is non-empty, its smallest element is called the
\emph{order of nilpotence of $\mathcal{W}$ with respect to $X$}.
\end{defi}
We will primarily use the fact that the Kulkarni–Nomizu product
is bilinear.
\begin{rem}
    If $h$ and $k$ are symmetric, then $h\owedge k$ possesses the
algebraic symmetries of the curvature tensor. Indeed, antisymmetry with respect to the first two variables and the
last two variables follows directly from
\eqref{eq:KN}. Symmetry under the exchange of the two pairs also
follows from the same formula. Finally, $(h\owedge k)_{ijkl}
+
(h\owedge k)_{iklj}
+
(h\owedge k)_{iljk}
=0,$

after grouping terms. Thus, $h\owedge k$ satisfies the
first Bianchi identity.
\end{rem}
\begin{lem}
For any symmetric tensors $h, k$ and any vector field $X$,
$\mathcal{L}_X(h\owedge k)
=
(\mathcal{L}_Xh)\owedge k
+
h\owedge(\mathcal{L}_Xk).
$
\end{lem}

\begin{proof}
It suffices to apply the Lie derivative to the definition of the Kulkarni--Nomizu product and use the Leibniz rule.
\end{proof}

\begin{ex}\label{exemple1}
  Let $I\subset\mathbb{R}$ be an open interval and consider the manifold
\[
M=I\times\mathbb{R}\times\mathbb{R}^{n-2},
\qquad n\geq 3.
\]
We denote the local coordinates by $(t,x,y_1,\ldots,y_{n-2}).$
Let $\mathfrak{f}:I\longrightarrow ]0,+\infty[$ be a smooth function. 
We consider the Riemannian metric on $M$ given by
\begin{equation}
\label{eq:metric-family}
g=dt^2+\mathfrak{f}(t)^2dx^2+\sum_{j=1}^{n-2}dy_j^2.
\end{equation}
Let us consider $\xi=\partial_t.$ Which means that
$\|\xi\|=g(\partial_t,\partial_t)=1\quad \text{et}\quad \xi^\flat=dt.$ 
Therefore, $$\eta:=\xi^\flat\otimes\xi^\flat
=dt\otimes dt.$$ 
The metric matrix is $(g_{ij})
=
\operatorname{diag}
\left(
1,\mathfrak{f}(t)^2,1,\ldots,1
\right)$ and its inverse matrix is $(g^{ij})
=
\operatorname{diag}
\left(
1,\frac{1}{\mathfrak{f}(t)^2},1,\ldots,1
\right).$

Let us recall the Christoffel symbol formula: $$\Gamma_{ij}^{k}
=\frac12g^{k\ell}
\left(
\partial_i g_{j\ell}
+\partial_j g_{i\ell}
-\partial_\ell g_{ij}
\right)$$
We see that $$\Gamma_{ix}^{k}
=\frac12g^{k\ell}
\left(
\partial_i g_{x\ell}
+\partial_x g_{i\ell}
-\partial_\ell g_{ix}
\right).$$ therefore $$\begin{cases}
    \Gamma_{ix}^{k}
=-\frac12g^{kt}\partial_t g_{ix}
\\
\Gamma_{ix}^{k}=\frac12g^{kx}\partial_i g_{xx}
\end{cases}\Leftrightarrow \begin{cases}
    \Gamma_{xx}^{t}
=-\frac12g^{tt}\partial_t g_{xx}
\\
\Gamma_{tx}^{x}=\frac12g^{xx}\partial_t g_{xx}
\end{cases}.$$
The only non-zero Christoffel symbols are: $$\begin{cases}
    \Gamma_{xx}^{t}
=- \mathfrak{f}'(t)\mathfrak{f}(t)
\\
\Gamma_{tx}^{x}=\frac{\mathfrak{f}'(t)}{\mathfrak{f}(t)}
\end{cases}.$$
So, $$\nabla_{\partial_t}\partial_t=0,\quad \nabla_{\partial_t}\partial_x=\frac{\mathfrak{f}'(t)}{\mathfrak{f}(t)}\partial_x,\quad \nabla_{\partial_x}\partial_t
=\frac{\mathfrak{f}'(t)}{\mathfrak{f}(t)}\partial_x, \quad \nabla_{\partial_x}\partial_x
=-\mathfrak{f}(t)\mathfrak{f}'(t)\partial_t$$
$$\text{et}\quad \nabla_{\partial_i}\partial_j=0 \quad \text{for the departments}\quad y_j.$$
Therefore, $\nabla_{\partial_t}
\nabla_{\partial_x}\partial_x
=\nabla_{\partial_t}(-\mathfrak{f}(t)\mathfrak{f}'(t)\partial_t)=
-(\mathfrak{f}'^2(t)+\mathfrak{f}(t)\mathfrak{f}''(t))\partial_t.$

On the other hand, $\nabla_{\partial_x}
\nabla_{\partial_t}\partial_x
=\nabla_{\partial_x}
\left(\frac{\mathfrak{f}'(t)}{\mathfrak{f}(t)}\partial_x\right)=\frac{a'(t)}{\mathfrak{f}(t)}\nabla_{\partial_x}\partial_x=
-\mathfrak{f}'^2(t)\partial_t.$

Thus, by applying \eqref{Rr}, we obtain: $$ R(\partial_t,\partial_x)\partial_x=-(\mathfrak{f}'^2(t)+\mathfrak{f}(t)\mathfrak{f}''(t))\partial_t+\mathfrak{f}'^2(t)\partial_t=-\mathfrak{f}(t)\mathfrak{f}''(t)\partial_t.$$
By the antisymmetry of the curvature tensor,
$$R(\partial_x,\partial_t)\partial_x
=\mathfrak{f}''(t)\mathfrak{f}(t)\partial_t.$$

Lowering the last index, we obtain: $$R(\partial_t,\partial_x,\partial_x,\partial_t)=g\left(R(\partial_t,\partial_x)\partial_x,
\partial_t
\right)=
-\mathfrak{f}''(t)\mathfrak{f}(t).$$

Thus, \begin{equation}\label{1er}
R(\partial_t,\partial_x,\partial_x,\partial_t)=-\mathfrak{f}''(t)\mathfrak{f}(t).
\end{equation}
The directions $\partial_{y_1},\ldots,\partial_{y_{n-2}}$ exhibit no curvature interaction with the directions $\partial_t$ and $\partial_x$.
$$\text{In particular, }\quad R(\partial_t,\partial_{y_j})=0,\quad R(\partial_x,\partial_{y_j})=0\quad \text{and}\quad R(\partial_{y_j},\partial_{y_i})=0.$$
Thus, the only non-trivial curvature direction is that of the plane
spanned by $\partial_t$ and $\partial_x.$

Let us evaluate the tensor $\eta\owedge g$ on $(\partial_t,\partial_x,\partial_x,\partial_t)$ by applying the Kulkarni-Nomizu product given by formula \eqref{produit}.

Since $$\eta(\partial_t,\partial_t)=1,
\quad
g(\partial_x,\partial_x)=\mathfrak{f}(t)^2,\qquad\text{and}\qquad \eta(\partial_t,\partial_x)=0,$$
we obtain \begin{equation}\label{2em}
    (\eta\owedge g)
(\partial_t,\partial_x,\partial_x,\partial_t)
=\eta(\partial_t,\partial_t)
g(\partial_x,\partial_x)=\mathfrak{f}^2(t).
\end{equation}

The expressions \eqref{1er} and \eqref{2em} directly yield the relation $R=
-\frac{\mathfrak{f}''(t)}{\mathfrak{f}(t)}\eta\owedge g.$

\end{ex}

We now assume that
$$R=\lambda\,\eta\owedge g,
\qquad
\eta=\xi^\flat\otimes\xi^\flat.$$

On the open set $M_\xi=\{p\in M:\xi_p\neq0\},$
let us set $u=\frac{\xi}{\|\xi\|}$
and
$\theta=u^\flat.$

Then $\xi^\flat\otimes\xi^\flat
=\|\xi\|^2\theta\otimes\theta.$

Setting $\kappa=-\lambda\|\xi\|^2,$
we can write, using our curvature convention, $$R=-\kappa(\theta\otimes\theta)\owedge g.$$

A contraction of \eqref{eq1} yields
\begin{equation}\label{eq:ricci}
\operatorname{Ric}=\lambda
\left[(2-n)\xi^\flat\otimes\xi^\flat
-\|\xi\|^2g
\right].
\end{equation}

Consequently,
\begin{equation}
\operatorname{Scal}
=
-2(n-1)\lambda\|\xi\|2.
\label{eq:scalar}
\end{equation}

On $M_\xi$, we thus obtain
$\operatorname{Ric}
=
\kappa g
+
(n-2)\kappa\,\theta\otimes\theta.
$

Thus:

\begin{equation}
\operatorname{Ric}(u,u)=(n-1)\kappa,
\end{equation}

and, for $Y\perp u$,
\begin{equation}
\operatorname{Ric}(Y,Y)=\kappa\|Y\|^2.
\end{equation}
\begin{prop}
On $M_\xi$, the Ricci tensor has at most two eigenvalues: $$(n-1)\kappa
\qquad\text{and}\qquad
\kappa.
$$ The eigenspace corresponding to $(n-1)\kappa$ is
$\operatorname{Span}\{\xi\}$, while $\xi^\perp$ is the eigenspace associated
with $\kappa$.
\end{prop}
\begin{proof}
From relation \eqref{eq:ricci}, $\operatorname{Ric}(\xi,\cdot)
=\lambda\left[(2-n)\|\xi\|^2\xi^\flat
-\|\xi\|^2\xi^\flat\right],$ 

whence $
\operatorname{Ric}(\xi,\cdot)
=
(1-n)\lambda\|\xi\|^2\xi^\flat=(n-1)\kappa \xi^\flat.$ Thus $\xi$ is an eigenvector direction.

Now, if $X\in\xi^\perp$, then $\eta(X)=0.$ Consequently, $\operatorname{Ric}(X,\cdot)
=
-\lambda\|\xi\|^2X^\flat=\kappa X^\flat.
$

Thus $\xi^\perp$ is also an eigenspace.
\end{proof}
\begin{rem}
    Let us choose locally an adapted orthonormal coordinate system $e_1=\frac{\xi}{\|\xi\|},
    \qquad
    e_2,\ldots,e_n\in\xi^\perp.$ With $\theta=\frac{\xi^\flat}{|\xi|}$, we have $\theta(e_1)=1,
    \qquad
    \theta(e_a)=0,
    \quad a\geq2.$ So under the hypothesis \eqref{eq1}, the sectional curvature is given by: 
    \begin{align*}
    K(e_1,e_a)
    &=
    -\lambda,
    \qquad a=2,\ldots,n,
    \label{eq:radialK}
    \\
    K(e_a,e_b)
    &=0,
    \qquad 2\leq a<b\leq n.
    \end{align*}
    Because for $a\geq2$, we have:

    \begin{align*}
    R_{1a1a}
    &=
    \lambda
    \left(
    \theta_1\theta_ag_{a1}
    +
    \theta_a\theta_1g_{1a}
    -
    \theta_1^2g_{aa}
    -
    \theta_a^2g_{11}
    \right)\\
    &=
    -\lambda.
    \end{align*}

    As the coordinate system is orthonormal, $K(e_1,e_a)=R_{1a1a}.$

    If $a,b\geq2$, then $\theta_a=\theta_b=0,$
    and all terms in $R_{abab}$ are zero. So $K(e_a,e_b)=0.$
\end{rem}
\begin{lem}\label{L1}
Suppose $n\geq3$ and the relation \eqref{eq:ricci}, we have:

$$-\mathrm{d} \kappa = \operatorname{div}(\lambda\xi^\flat)\xi^\flat + \lambda\,\nabla_\xi\xi^\flat.$$

Equivalently, for any field $X$,
\begin{equation}
X(-\kappa)
=
\operatorname{div}(\lambda\xi^\flat),g(\xi,X)
+
\lambda,g(\nabla_\xi\xi,X).
\end{equation}
\end{lem}
\begin{proof}
According to the relation \eqref{eq:ricci}, we have: $\operatorname{Ric} = (2-n)\lambda\xi^\flat\otimes\xi^\flat + \kappa g.$
SO
\begin{align*}
\operatorname{div}(\operatorname{Ric})
={}&
(2-n)\operatorname{div}(\lambda\xi^\flat\otimes\xi^\flat) + \mathrm{d}\kappa.
\end{align*}

For the $1$-form $\xi^\flat$, the divergence of the tensor product verifies:
$$\operatorname{div}(\lambda\xi^\flat\otimes\xi^\flat)=\operatorname{div}(\lambda\xi^\flat)\xi^\flat+\lambda\nabla_\xi\xi^\flat.$$
So,
$$\operatorname{div}(\operatorname{Ric}) = (2-n)\left[ \operatorname{div}(\lambda\xi^\flat)\xi^\flat+\lambda\nabla_\xi\xi^\flat\right]+\mathrm{d}\kappa.$$

On the other hand, the scalar curvature relation gives $\frac{1}{2}\operatorname{d}\mathrm{Scal} = (n-1)\mathrm{d}\kappa.$

The second Bianchi identity, $\operatorname{div}(\operatorname{Ric}) = \frac{1}{2}\operatorname{d}\mathrm{Scal}$, then leads to:
$$(2-n) \left[ \operatorname{div}(\lambda\xi^\flat)\xi^\flat + \lambda\nabla_\xi\xi^\flat \right] +\mathrm{d}\kappa = (n-1)\mathrm{d}\kappa.$$

By grouping the terms into $\mathrm{d}\kappa$, it comes:
$$(2-n)\left[ \mathrm{d}\kappa + \operatorname{div}(\lambda\xi^\flat)\xi^\flat + \lambda\nabla_\xi\xi^\flat \right] = 0.$$

As $n\geq3$, we deduce:
$$-\mathrm{d}\kappa = \operatorname{div}(\lambda\xi^\flat)\xi^\flat + \lambda\nabla_\xi\xi^\flat.$$

Finally, by evaluating this relation on a field of vectors $X$ and noting that $(\nabla_\xi\xi^\flat)(X) = g(\nabla_\xi\xi, X)$ with $\xi^\flat(X) = g(\xi, X)$, we obtain the second formulation.
\end{proof}

\begin{prop}
\label{prop:derivees}
For any integer $k\geq0$ and any differentiable tensor $\mathcal{W}$,
\begin{equation}
    \frac{d^k}{dt^k}\Phi_t^*\mathcal{W}= \Phi_t^*(\mathcal{L}_X^k\mathcal{W}).
    \label{eq:derivees-successives}
\end{equation}
\end{prop}

\begin{proof}
For $k=0$, the identity is immediate since $\mathcal{L}_X^0\mathcal{W}=\mathcal{W}.$

Now assume that $ \frac{d^k}{dt^k}\Phi_t^*\mathcal{W}=\Phi_t^*(\mathcal{L}_X^k\mathcal{W}).$

Differentiating once more, we obtain
$$ \frac{d^{k+1}}{dt^{k+1}}\Phi_t^*\mathcal{W}= \frac{d}{dt} \left[  \Phi_t^*(\mathcal{L}_X^k\mathcal{W}) \right]=\Phi_t^* \left(\mathcal{L}_X(\mathcal{L}_X^k\mathcal{W})\right)=\Phi_t^*\left(\mathcal{L}_X^{k+1}\mathcal{W}\right).$$

The property therefore holds for all $k\geq0$.
\end{proof}
\begin{defi}
    Let $(M,g)$ be a Riemannian manifold of dimension $n\ge 2$, $X$ a smooth vector field, $\mu$ and $\mu_1$ real constants, $f$ a smooth function, and $\theta$ a $1$-form.
    \begin{enumerate}
        \item The quadruple $(M,g,X,\mu)$ is called a Ricci soliton if \begin{equation}\label{e0}
            \operatorname{Ric}+\frac{1}{2}\mathcal{L}_Xg=\mu g;        \end{equation}
        \item The quadruple $(M,g,X,f)$ is called an almost Ricci soliton if \begin{equation}
            \operatorname{Ric}+\frac{1}{2}\mathcal{L}_Xg=f g;        \end{equation}
        \item The quintuple $(M,g,X,f,\mu_1)$ is called a $\theta$-almost Yamabe soliton if \begin{equation}\label{e00}
           \frac{1}{2}\mathcal{L}_Xg=\left(f-\operatorname{Scal}\right)g+\mu_1\theta\otimes\theta.        \end{equation}
    \end{enumerate}
\end{defi}
\begin{rem}
Regarding the nature of the geometric soliton determined by the constant $\mu$, we distinguish three regimes based on its sign:
\begin{itemize}[label=$\bullet$]
    \item If $\mu> 0$, the soliton is called \textbf{shrinking};
    \item If $\mu < 0$, the soliton is called \textbf{expanding};
    \item If $\mu = 0$, the soliton is called \textbf{stationary}.
\end{itemize}
\end{rem}
\section{Main results}
In this section, we present the main results of our study concerning the characterization of geometric solitons and finite-order symmetries on Riemannian manifolds with rank-one anisotropic curvature. Throughout the remainder of this text, $(M,g)$ denotes an $n$-dimensional Riemannian manifold whose Riemann curvature tensor satisfies condition \eqref{eq1}, and $\xi \in \mathfrak{X}(M)$ is a vector field defining the characteristic distribution $\mathcal{D} = \operatorname{Span}\{\xi\}$.
\subsection{Study of the Ricci soliton case}
\begin{thm}
Let $(M, g)$ be a Riemannian manifold of dimension $n \geq 3$ whose curvature tensor is given by the relation \eqref{eq1}, and let $\xi$ be a recurrent vector field such that $\nabla_X \xi = \omega(X)\xi$, where the recurrence 1-form is of the form $\omega = (n-2)\lambda \xi^\flat$.
 Then, the expression for the Ricci tensor yields an almost Ricci soliton of the form $\left(M, g, \frac{\nabla(\ln(\|\xi\|))}{(n-2)\lambda}, -\lambda \|\xi\|^2\right)$.
\end{thm}

\begin{proof}

According to Koszul's formula, the Lie derivative of the metric with respect to the field $\xi$ is expressed as:
$$(\mathcal{L}_\xi g)(X, Y) = g(\nabla_X \xi, Y) + g(X, \nabla_Y \xi)$$
Substituting the recurrence condition $\nabla_X \xi = (n-2)\lambda(X)\xi = (n-2)\lambda \, g(\xi, X)\xi$, we obtain:
\begin{align*}
(\mathcal{L}_\xi g)(X, Y) &= g( (n-2)\lambda g(\xi, X)\xi, Y) + g(X,  (n-2)\lambda g(\xi, Y)\xi) \\
&=  (n-2)\lambda g(\xi, X)g(\xi, Y) +  (n-2)\lambda g(\xi, Y)g(\xi, X) \\
&= 2 (n-2)\lambda g(\xi, X)g(\xi, Y)
\end{align*}
In tensor notation, this is rewritten as:
$$\mathcal{L}_\xi g = 2 (n-2)\lambda (\xi^\flat \otimes \xi^\flat)$$

The expression for the Ricci tensor given in \eqref{eq:ricci} thus becomes \begin{equation}\label{a}
    \operatorname{Ric}=-\frac{1}{2}\mathcal{L}_\xi g-\lambda\|\xi\|^2g.
\end{equation}

As is the case for any recurrent field, the 1-form $\omega$ is necessarily exact and is expressed as $\omega = \mathrm{d}(\ln \|\xi\|)$. Using the hypothesis $\omega = (n-2)\lambda \xi^\flat$, we have :
$$\mathrm{d}(\ln \|\xi\|) =(n-2)\lambda  \xi^\flat$$
Applying the musical isomorphism (the gradient) to pass from 1-forms to vector fields, we obtain:
$$\nabla(\ln \|\xi\|) = (n-2)\lambda \xi.$$ Thus, equation \eqref{a} becomes $$\operatorname{Ric}+\frac{1}{2}\mathcal{L}_{\frac{\nabla(\ln \|\xi\|) }{(n-2)\lambda }}g=-\lambda\|\xi\|^2g$$
\end{proof}
\begin{thm}
   Let $(M,g)$ be a Riemannian manifold of dimension $n \geq 3$ that is not Ricci-flat. If the quadruple $(M,g,X,\mu)$ is a Ricci soliton, then the vector field $X$ is not conformal. Furthermore, the Riemannian manifold $(M,g)$ cannot be an Einstein manifold.
\end{thm}
\begin{proof}
    The condition that $(M,g,X,\mu)$ is a Ricci soliton is equivalent to equation \eqref{e0}.
    
    Using relation \eqref{eq:ricci}, equation \eqref{e0} becomes \begin{equation}\label{e}
        \lambda
\left[(2-n)\xi^\flat\otimes\xi^\flat-\|\xi\|^2g
\right]+\frac{1}{2}\mathcal{L}_Xg=\mu g\Leftrightarrow \mathcal{L}_Xg =2\left[\lambda\left[
(2-n)\xi^\flat\otimes\xi^\flat\right] +
\left(\mu-
\lambda\|\xi\|^2\right)g
\right]  
    \end{equation} Suppose that $\mathcal{L}_Xg=2\varphi g$; we then obtain $$\left(\varphi- \mu+
\lambda\|\xi\|^2\right)g=\lambda\left[
(2-n)\xi^\flat\otimes\xi^\flat\right].$$ Since $\xi^\flat\otimes\xi^\flat$ has rank 1, $\lambda \neq 0$, and $g$ has rank $n$, the preceding equality cannot hold. Furthermore, if we set $\operatorname{Ric}=\beta g$, we obtain $(2-n)\lambda\xi^\flat\otimes\xi^\flat=(\beta+\lambda\|\xi\|^2)g$. Given that $n \neq 2$ and $\lambda \neq 0$ (since $\operatorname{Ric} \neq 0$), the equality does not hold, yielding the result. 
\end{proof} 
\begin{thm}
   Let $(M,g)$ be a Riemanian manifold of dimension $2$. The quintuple $(M,g,X,\mu)$ is a Ricci soliton if and only if $(M,g,X,-\mu)$ is a Yamabe soliton. Moreover if $\|\xi\|=c^2$ then it is an Einstein variety.  
\end{thm}
\begin{proof}
    Suppose $(M,g,X,\mu)$ is Ricci soliton so the equation \eqref{e} becomes $\mathcal{L}_Xg=2\left(\mu-
\lambda\|\xi\|^2\right)g$, since $\operatorname{Scal}=-
\lambda\|\xi\|^2$ so we just need to set $\rho=-\mu$ and we obtain the result.
\end{proof}
\begin{cor}
    Let $(M,g)$ be a Riemannian variety whose curvature tensor satisfies the relation \eqref{eq1}, if the quadruplet 
 $(M,g,\xi,\mu)$ is a Ricci soliton then $$\frac{1}{2(1-n)}\left(\nabla\left(\operatorname{div}(\xi)\right)\right)^\flat=\operatorname{div}(\lambda\xi^\flat)\xi^\flat+\lambda\nabla_\xi\xi^\flat.$$
\end{cor}
\begin{proof}
    Suppose that $(M,g,\xi,\mu)$ is a Ricci soliton so the relation \eqref{e} becomes $$\mathcal{L}_\xi g =2\left[\lambda\left[
(2-n)\xi^\flat\otimes\xi^\flat\right] +
\left(\mu-
\lambda|\xi|^2\right)g
\right] \Rightarrow \operatorname{div}(\xi)=n\mu +2(1-n)\lambda\|\xi\|^2.$$ Since $\kappa=-\lambda\|\xi\|^2$ therefore $\mathrm{d}\left(\operatorname{div}(\xi)\right)=2(n-1)\mathrm{d}\kappa$.
According to Lemma \ref{L1} we have $$\mathrm{d}\left(\operatorname{div}(\xi)\right)=2(1-n)\left(\operatorname{div}(\lambda\xi^\flat)\xi^\flat+\lambda\nabla_\xi\xi^\flat\right)$$ or $\mathrm{d}\left(\operatorname{div}(\xi)\right)=\left(\nabla\left(\operatorname{div}(\xi)\right)\right)^\flat$ hence the result.
\end{proof}
\begin{prop}
    Let $X$ be a smooth vector field of the variety $(M,g)$ of dimension $n\geq 2$. If the scalar curvature $\operatorname{Scal}$ is a constant denoted $\operatorname{Scal}_0$ and $\lambda$ a positive smooth function then the quadruplet $(M,g,X,\mu)$ is a Ricci soliton if and only if $\Big(M,g,X,\mu+\frac{(2n-1)\operatorname{Scal}_0}{2(n-1)},(2-n)\Big)$ is a Yamabe $\sqrt{\lambda}\xi^\flat$-soliton.
\end{prop}
\begin{proof}
    Suppose that the quadruplet $(M,g,X,\mu)$ is a Ricci soliton and $\operatorname{Scal}=\operatorname{Scal}_0\in\mathbb{R}$, so $\lambda\|\xi\|^2=\frac{\operatorname{Scal_0}}{2(1-n)}$ and the relation \eqref{e}, we obtain $$\frac{1}{2}\mathcal{L}_Xg=\lambda\left[
(2-n)\xi^\flat\otimes\xi^\flat\right] +
\left(\mu+
\frac{\operatorname{Scal_0}}{2(n-1)}\right)g$$ $$\Leftrightarrow \frac{1}{2}\mathcal{L}_Xg=\left(\mu+
\frac{\operatorname{(2n-1)Scal_0}}{2(n-1)}-\operatorname{Scal}_0\right)g+
(2-n)\Big(\sqrt{\lambda}\xi^\flat\Big)\otimes\Big(\sqrt{\lambda}\xi^\flat\Big)$$
\end{proof}
\begin{thm}\label{T1}
Let $(M,g)$ be a Riemannian manifold whose curvature tensor satisfies relation \eqref{eq:ricci}; the quadruple $(M,g,X,\mu)$ is a Ricci soliton if and only if
$$\mathcal{L}_XR
=
\left[
X(\lambda)
+
2\lambda(\mu-\lambda\|\xi\|^2)
\right]
(\xi^\flat\otimes\xi^\flat)\owedge g+
\lambda\,
\mathcal{L}_X
(\xi^\flat\otimes\xi^\flat)\owedge g.
$$

\end{thm}

\begin{proof}
Assume that $(M,g,X,\mu)$ is a Ricci soliton; we thus directly obtain relation \eqref{e}.

Since $R$ is given by relation \eqref{eq1}, the Lie derivative yields \begin{equation}\label{R}
    \mathcal{L}_XR=X(\lambda)\eta\owedge g+
\lambda(\mathcal{L}_X\eta)\owedge g+\lambda\eta\owedge\mathcal{L}_Xg.
\end{equation}
For this reason 
$\mathcal{L}_XR=X(\lambda)\eta\owedge g+
\lambda(\mathcal{L}_X\eta)\owedge g+2\lambda\eta\owedge
\left[
\lambda(2-n)\eta
+
(\mu-\lambda\|\xi\|^2)g
\right].$  

 $$\text{Thus,}\quad\mathcal{L}_XR=X(\lambda)\eta\owedge g+\lambda(\mathcal{L}_X\eta)\owedge g+
2\lambda^2(2-n)\eta\owedge\eta+2\lambda(\mu-\lambda\|\xi\|^2)\eta\owedge g.$$
Now, $\eta\owedge\eta=0,$
since $\eta=\xi^\flat\otimes\xi^\flat$
is a symmetric tensor of rank $1$. So it comes $$\mathcal{L}_XR=\left[
X(\lambda)+2\lambda(\mu-\lambda\|\xi\|^2)
\right]\eta\owedge g+\lambda(\mathcal{L}_X\eta)\owedge g.$$
Returning to $\eta=\xi^\flat\otimes\xi^\flat,$
we finally obtain
$$\mathcal{L}_XR=\left[X(\lambda)+2\lambda(\mu-\lambda\|\xi\|^2)\right]
(\xi^\flat\otimes\xi^\flat)\owedge g+\lambda\mathcal{L}_X(\xi^\flat\otimes\xi^\flat)\owedge g.$$
\end{proof}

\begin{prop}\label{P1}
Let $(M^n,g)$ be a Riemannian manifold of dimension $n\geq 3$ that is not Ricci-flat.
Suppose that the curvature tensor of $g$ is given by the relation \eqref{eq1}.

If the quadruple $(M,g,X,\mu)$ is a Ricci soliton and $\mathcal{L}_XR=0\quad(
\mathcal{L}_X\operatorname{Ric}=0),$
then $(M,g,X,\mu)$ becomes stationary. Furthermore,

$$\mathcal{L}_X\eta
=
-\left(
\frac{X(\lambda)}{\lambda}
-2\lambda\|\xi\|^2
\right)\eta,\quad
X(\lambda\|\xi\|^2)=2\lambda^2\|\xi\|^4\qquad \text{and}\qquad
X(\operatorname{Scal})
=
-\frac{1}{n-1}\operatorname{Scal}^2$$
and on any open set where $\operatorname{Scal}\neq0$, we therefore have $
X\left(\frac{1}{\operatorname{Scal}}\right)
=
\frac{1}{n-1}.$
\end{prop}
\begin{proof}
Suppose that $(M,g,X,\mu)$ is a Ricci soliton. According to Theorem \ref{T1},
we have: $$
\mathcal{L}_XR=\left[X(\lambda)+2\lambda(\mu-\lambda\|\xi\|^2)\right]\eta\owedge g+\lambda(\mathcal{L}_X\eta)\owedge g.$$
By the bilinearity of the Kulkarni--Nomizu product,
$$\mathcal{L}_XR=\left\{\lambda\mathcal{L}_X\eta+\left[X(\lambda)+
2\lambda(\mu-\lambda\|\xi\|^2)\right]\eta\right\}\owedge g.$$
Since $n\geq3$, the map $\zeta\longmapsto \zeta\owedge g$
is injective on the space of symmetric tensors. The hypothesis
$\mathcal{L}_XR=0$ thus yields$$\lambda\mathcal{L}_X\eta
+
\left[
X(\lambda)
+
2\lambda(\mu-\lambda|\xi|^2)
\right]\eta
=
0.$$
Since $\lambda\neq0$,

\begin{equation}\label{proportionelle}
\mathcal{L}_X\eta=-\left[\frac{X(\lambda)}{\lambda}+2(\mu-\lambda\|\xi\|^2)\right]\eta.
\end{equation}

Let us now compute $\mathcal{L}_X\operatorname{Ric}$. We have
$\operatorname{Ric}=\lambda[(2-n)\eta-\|\xi\|^2g],$
whence
$$\mathcal{L}_X\operatorname{Ric}=X(\lambda)[(2-n)\eta-\|\xi\|^2g]+\lambda(2-n)\mathcal{L}_X\eta-\lambda X(\|\xi\|^2)g
-\lambda\|\xi\|^2\mathcal{L}_Xg.$$ 
Substituting \eqref{e} and \eqref{proportionelle}, we obtain:
\begin{align*}
    \mathcal{L}_X\operatorname{Ric}
&=X(\lambda)[(2-n)\eta-\|\xi\|^2g]
+\lambda(n-2)\left[\frac{X(\lambda)}{\lambda}+2(\mu-\lambda\|\xi\|^2)
\right]\eta-\lambda X(\|\xi\|^2)g\\
&-2\lambda\|\xi\|^2\left[\lambda
(2-n)\eta +\left(\mu-\lambda\|\xi\|^2\right)g
\right].
\end{align*} By simplifying, we obtain
$$\mathcal{L}_X\operatorname{Ric}
=-2(n-2)\lambda\mu\eta-\left[X(\lambda)\|\xi\|^2+\lambda X(\|\xi\|^2)+2\lambda\|\xi\|^2
(\mu-\lambda\|\xi\|^2)
\right]g.$$
Since $\mathcal{L}_X\operatorname{Ric}=0,$
and given that $n\geq3$, $\lambda\neq0$, and $\eta$ is not proportional
to $g$, the two components must vanish separately. The component
along $\eta$ yields $-2(n-2)\lambda\mu=0.$
Consequently, $\mu=0.$ The soliton is therefore stationary.

The component along $g$ then becomes $X(\lambda)\|\xi\|^2+\lambda X(\|\xi\|^2)
-2\lambda^2\|\xi\|^4=0.$

Or $X(\lambda\|\xi\|^2)=X(\lambda)\|\xi\|^2
+\lambda X(\|\xi\|^2).$ Thus, $X(\lambda\|\xi\|^2)=2\lambda^2\|\xi\|^4.$

Finally, by contraction of the Ricci tensor, $\operatorname{Scal}
=-2(n-1)\lambda\|\xi\|^2.$
Differentiating with respect to $X$, we obtain $X(\operatorname{Scal})
=-2(n-1)X(\lambda\|\xi\|^2).$
Therefore,  $X(\operatorname{Scal})
=-4(n-1)\lambda^2\|\xi\|^4.$
Since $\operatorname{Scal}^2
=
4(n-1)^2\lambda^2\|\xi\|^4.$
We deduce $X(\operatorname{Scal})
=
-\frac{1}{n-1}\operatorname{Scal}^2.$

On the open set where  $\operatorname{Scal}\neq0$, so $X\left(\frac{1}{\operatorname{Scal}}\right)
=-\frac{X(\operatorname{Scal})}{\operatorname{Scal}^2}
=\frac{1}{n-1}.$
\end{proof}

\begin{cor}
Under the hypotheses of Proposition \ref{P1}, if $\gamma$ is an integral curve of $X$, then on any open set where $\operatorname{Scal}(\gamma(t)) \neq 0$, we have $\operatorname{Scal}(\gamma(t)) = \frac{n-1}{t+a}$ for some $a \in \mathbb{R}$.
\end{cor}

\begin{proof}
By Proposition \ref{P1}, $X\left(\frac{1}{\operatorname{Scal}}\right) = \frac{1}{n-1}$.

Along an integral curve $\gamma$ of $X$,
$$\frac{d}{dt} \left( \frac{1}{\operatorname{Scal}(\gamma(t))} \right) = X\left(\frac{1}{\operatorname{Scal}}\right) \bigg|_{\gamma(t)} = \frac{1}{n-1}.$$
Integrating yields $\frac{1}{\operatorname{Scal}(\gamma(t))} = \frac{t}{n-1} + C$,
whence $\operatorname{Scal}(\gamma(t)) = \frac{n-1}{t+(n-1)C}$. It suffices to set $a=(n-1)C$.
\end{proof}

\begin{prop}
Let $f$ be a smooth function on the manifold $(M,g)$ of dimension $n \geq 2$. If the quadruple $(M,g,\nabla f,\mu)$ is a gradient Ricci soliton such that:
\begin{enumerate}
    \item $f$ is harmonic, then $\operatorname{div}\left(\frac{\xi^\flat}{\|\xi\|^2}\right)\xi^\flat = -\frac{\nabla_\xi\xi^\flat}{\|\xi\|^2}$;
 \item If $f$ is Lipschitz, then the function $(x_1, \dots, x_n) \mapsto \lambda(x_1, \dots, x_n)\|\xi(x_1, \dots, x_n)\|^2$ is bounded if and only if the soliton is stationary, or if $f$ is bounded, or if $f$ is positive and bounded above, or if $f$ is negative and bounded below.
\end{enumerate}
\end{prop}
\begin{proof}
    Consider a smooth function $f$ on a manifold $(M,g)$ of dimension $n \geq 2$ such that $(M,g,\nabla f,\mu)$ is a gradient Ricci soliton.
    \begin{enumerate}
        \item Suppose that $f$ is harmonic; then equation \eqref{e} becomes $\nabla^2f=\lambda\left[(2-n)\xi^\flat\otimes\xi^\flat\right] + \left(\mu-\lambda\|\xi\|^2\right)g$. Applying the trace operator yields $n\mu+2(1-n)\lambda|\xi|^2=0$, which implies $\kappa=-\lambda\|\xi\|^2=\frac{\mu n}{2(1-n)}$; thus, Lemma \ref{L1} becomes $0=\operatorname{div}\left(\lambda\xi^\flat\right)\xi^\flat+\lambda\nabla_\xi\xi^\flat$. Substituting the expression for $\lambda$ yields $$\operatorname{div}\left(\frac{\xi^\flat}{\|\xi\|^2}\right)\xi^\flat=-\frac{\nabla_\xi\xi^\flat}{\|\xi\|^2}$$
        \item Suppose that the function $f$ is Lipschitz.
        
    According to Hamilton's identity, there exists a real constant $b$ such that: $$ \operatorname{Scal}= b- \vert{}\nabla f\vert{}^2 + 2\mu f.$$ Since $f$ is Lipschitz, there exists a positive real constant $M$ such that $\|\nabla f\|\leq M$; thus, the triangle inequality yields $$|\operatorname{Scal}|\leq |b|+M^2+2|\mu f|.$$ As $\operatorname{Scal}=-2(n-1)\lambda\|\xi\|^2$, we have \begin{equation}\label{H}
    \left|\lambda\|\xi\|^2\right|\leq \frac{1}{2(n-1)}\left(|b|+M^2+2|\mu f|\right)
\end{equation}
\begin{enumerate}
    \item If $\mu=0$, then inequality \eqref{H} becomes $\left|\lambda\|\xi\|^2\right|\leq \frac{1}{2(n-1)}\left(|b|+M^2\right)$;
    \item if there exists a positive real constant $p$ such that $|f|\leq p$, then inequality \eqref{H} becomes: $\left|\lambda\|\xi\|^2\right|\leq \frac{1}{2(n-1)}\left(|b|+M^2+2p|\mu|\right)$;
    \item if $f$ is positive and bounded above by a positive real constant $p_1$, then inequality \eqref{H} becomes: $\left|\lambda\|\xi\|^2\right|\leq \frac{1}{2(n-1)}\left(|b|+M^2+2p_1|\mu|\right)$;
     \item if $f$ is negative and bounded below by a negative real constant $p_2$, then inequality \eqref{H} becomes: $\left|\lambda\|\xi\|^2\right|\leq \frac{1}{2(n-1)}\left(|b|+M^2-2p_2|\mu|\right)$.
\end{enumerate} 
\end{enumerate}
\end{proof} 
\begin{prop}
Let $(M^n,g)$, $n\geq2$, be a Riemannian variety satisfying the relation \eqref{eq1}. Then
 $\xi$ is not a closed conformal vector field.
\end{prop}
\begin{proof}
Let us assume the condition $\nabla_X\xi=\varphi X$ on $M_\xi$, we therefore have:

$$R(X,Y)\xi
=
\nabla_X\nabla_Y\xi
-\nabla_Y\nabla_X\xi
-\nabla_{[X,Y]}\xi=
\nabla_X(\varphi Y)
-\nabla_Y(\varphi X)
-\varphi[X,Y]=
X(\varphi)Y-Y(\varphi)X.$$
So,
\begin{equation}
R(X,Y)\xi
=
X(\varphi)Y-Y(\varphi)X.
\label{eq:curv-xi-conforme}
\end{equation}

On the other hand, by hypothesis,
\[R
=
\lambda
(\xi^\flat\otimes\xi^\flat)\owedge g.
\]
For all $X,Y,Z\in\mathfrak{X}(M_\xi)$, we have
$$g(R(X,Y)\xi,Z)
=
R(X,Y,\xi,Z)=
\lambda
\bigl[
(\xi^\flat\otimes\xi^\flat)\owedge g
\bigr](X,Y,\xi,Z).$$
However,\begin{align*}
    (\xi^\flat\otimes\xi^\flat)\owedge g
\bigr](X,Y,\xi,Z)&=
g(\xi,X)g(\xi,Z)g(Y,\xi)
+|\xi|^2g(X,Z)g(\xi,Y)\\
&-g(\xi,X)g(\xi,Y)g(\xi,Z)
-g(\xi,Y)g(\xi,Z)g(\xi,X)=0.
\end{align*}
Consequently,
\begin{equation}
R(X,Y)\xi=0.
\label{eq:Rxi-zero}
\end{equation}

Comparing \eqref{eq:curv-xi-conforme} and \eqref{eq:Rxi-zero}, we obtain
\[
X(\varphi)Y-Y(\varphi)X=0,
\qquad \forall X,Y.
\]
Thus,
\[
d\varphi=0.
\]

Furthermore, \eqref{eq:Rxi-zero} implies, upon contraction,
\[
\operatorname{Ric}(\xi,X)=0,
\qquad \forall X.
\]

However, contracting the tensor $R=
\lambda
(\xi^\flat\otimes\xi^\flat)\owedge g$ yields $\operatorname{Ric}=\lambda
\left[
(2-n)\xi^\flat\otimes\xi^\flat
-\|\xi\|^2g
\right].$
Evaluating on $(\xi,X)$, we obtain
\begin{align*}
\operatorname{Ric}(\xi,X)
&=\lambda\left[
(2-n)|\xi|^2g(\xi,X)
-\|\xi\|^2g(\xi,X)
\right]\\
&=
\lambda(1-n)\|\xi\|^2g(\xi,X).
\end{align*}
Since $\operatorname{Ric}(\xi,X)=0$, it follows that
$\lambda(1-n)\|\xi\|^2g(\xi,X)=0,
\qquad \forall X.$
Taking $X=\xi$, we obtain $\lambda(1-n)\|\xi\|^4=0.$ On $M_\xi$, we have $\|\xi\|^2>0$, and since $n\geq2$, $1-n\neq0.$
Consequently, $\lambda=0
\qquad\text{on } M_\xi.$
This is a contradiction.
\end{proof}
\begin{prop}\label{PP1}
  Let $(M^n, g)$, $n \ge2$, be a Riemannian manifold satisfying the relation \eqref{eq:ricci}. Then, for any non-negative integer $m$ and any smooth vector field $X$, we have: \begin{equation}
       \mathcal{L}^{m}_X\operatorname{Ric}= (2-n)
\sum_{j=0}^{m}
\sum_{i=0}^{j}
\binom{m}{j}\binom{j}{i}
X^{m-j}(\lambda)\times
\left(
\mathcal{L}_X^{j-i}\xi^\flat
\otimes
\mathcal{L}_X^i\xi^\flat
\right)+ \mathcal{Z}_mg
   \end{equation} where  $X^{m-j}(\lambda)=\mathcal{L}^{m-j}_X\lambda$, $\mathcal{Z}_m=X(\mathcal{Z}_{m-1})+\mathcal{Z}_0\mathcal{Z}_{m-1}$ for all $m\geq 1$ and $\mathcal{Z}_0=-2\lambda \|\xi\|^2$.
\end{prop}
\begin{proof}
We have  $\operatorname{Ric}=\lambda
\left[
(2-n)\xi^\flat\otimes\xi^\flat
-\|\xi\|^2g
\right]$   so $$\mathcal{L}^m\operatorname{Ric}=\mathcal{L}^m\left(\lambda(2-n)\xi^\flat\otimes\xi^\flat\right)
+\mathcal{L}^m\left(-
2\lambda\|\xi\|^2g\right)$$
 Let us prove by induction on $m\in\mathbb{N}$ that, for any scalar function $\lambda$ and any vector field $\xi$,

\begin{equation}\label{1}
    \mathcal{L}_X^m
\left(
\lambda(2-n)\,\xi^\flat\otimes\xi^\flat
\right)
+\mathcal{L}^m\left(-
2\lambda\|\xi\|^2g\right)=
(2-n)
\sum_{j=0}^{m}
\sum_{i=0}^{j}
\binom{m}{j}\binom{j}{i}
X^{m-j}(\lambda)\times
\left(
\mathcal{L}_X^{j-i}\xi^\flat
\otimes
\mathcal{L}_X^i\xi^\flat
\right)+\mathcal{Z}_mg.
\end{equation}

\begin{enumerate}
    \item For $m=0$, the left-hand side of \eqref{1} is

$$
\lambda(2-n)\,\xi^\flat\otimes\xi^\flat.
-2\lambda|\xi|^2g.$$
On the other hand, the right-hand member becomes

$$
(2-n)
\binom00\binom00
X^0(\lambda)
\left(
\mathcal{L}_X^0\xi^\flat
\otimes
\mathcal{L}_X^0\xi^\flat
\right)+\mathcal{Z}_0g,
$$

that's to say

$$
(2-n)\lambda\,\xi^\flat\otimes\xi^\flat-2\lambda\|\xi\|^2g.
$$

The two sides are therefore equal. The property holds for $m=0$.

\item Suppose that formula \eqref{1} holds for an integer $m \geq 0$. Let us show that it then holds for $m+1$.

Applying $\mathcal{L}_X$ to both sides of \eqref{1}, we obtain

\begin{align*}
     \mathcal{L}_X^{m+1}
\left(
\lambda(2-n)\,\xi^\flat\otimes\xi^\flat
\right)
+\mathcal{L}^{m+1}\left(-
2\lambda|\xi|^2g\right)&=
(2-n)
\sum_{j=0}^{m+1}
\sum_{i=0}^{j}
\binom{m+1}{j}\binom{j}{i}
X^{m+1-j}(\lambda)\\
&\times
\left(
\mathcal{L}_X^{j-i}\xi^\flat
\otimes
\mathcal{L}_X^i\xi^\flat
\right)+\mathcal{Z}_{m+1}g.
\end{align*}

The Lie derivative of a scalar function is given by: $$\mathcal{L}_X f=X(f),$$
and using Leibniz's rule and Leibniz's rule for the tensor product, we have
$$
\begin{aligned}
&\mathcal{L}_X
\left[
X^{m-j}(\lambda)
\left(
\mathcal{L}_X^{j-i}\xi^\flat
\otimes
\mathcal{L}_X^i\xi^\flat
\right)
\right]=
X^{m-j+1}(\lambda)
\left(
\mathcal{L}_X^{j-i}\xi^\flat
\otimes
\mathcal{L}_X^i\xi^\flat
\right)+
X^{m-j}(\lambda)
\left(
\mathcal{L}_X^{j-i+1}\xi^\flat
\otimes
\mathcal{L}_X^i\xi^\flat
\right)\\
&\quad+
X^{m-j}(\lambda)
\left(
\mathcal{L}_X^{j-i}\xi^\flat
\otimes
\mathcal{L}_X^{i+1}\xi^\flat
\right).
\end{aligned}
$$

Thus,
\begin{align*}
\mathcal{L}_X^{m+1}
\left(
\lambda(2-n)\,\xi^\flat\otimes\xi^\flat
\right)&=(2-n)
\sum_{j=0}^{m}
\sum_{i=0}^{j}
\binom{m}{j}\binom{j}{i}
X^{m+1-j}(\lambda)\times
\left(
\mathcal{L}_X^{j-i}\xi^\flat
\otimes
\mathcal{L}_X^i\xi^\flat
\right)\\
&+(2-n)
\sum_{j=0}^{m}
\sum_{i=0}^{j}
\binom{m}{j}\binom{j}{i}
X^{m-j}(\lambda)\times
\left(
\mathcal{L}_X^{j-i+1}\xi^\flat
\otimes
\mathcal{L}_X^i\xi^\flat
\right)\\
+&(2-n)
\sum_{j=0}^{m}
\sum_{i=0}^{j}
\binom{m}{j}\binom{j}{i}
X^{m-j}(\lambda)\times
\left(
\mathcal{L}_X^{j-i}\xi^\flat
\otimes
\mathcal{L}_X^{i+1}\xi^\flat
\right).
\end{align*}
Let us now perform the appropriate changes of indices. For the last two summations in the expression for $\mathcal{L}_X^{m+1} \left( \lambda(2-n)\,\xi^\flat\otimes\xi^\flat \right)$ found above, we set $j'=j+1$ and $i'=i+1$, respectively. After reindexing and grouping the terms, the binomial coefficients satisfy

$$\binom{m}{j}\binom{j}{i}+\binom{m}{j-1}\binom{j-1}{i}+\binom{m}{j-1}\binom{j-1}{i-1}=\binom{m+1}{j}\binom{j}{i}.$$

Indeed,
$$\binom{m}{j}\binom{j}{i}=\frac{m!}{i!(j-i)!(m-j)!},$$

and the preceding identity is a consequence of Pascal's formula applied successively to the binomial coefficients.
and for the remaining term, we have $$\mathcal{L}^{m+1}_X(-2\lambda\|\xi\|^2g)=\mathcal{L}_X \left(\mathcal{L}^m_X(-2\lambda\|\xi\|^2g) \right)=\mathcal{L}_X \left( \mathcal{Z}_m g \right) = \left( X(\mathcal{Z}_m) \right) g + \mathcal{Z}_m \mathcal{L}_X g$$

By replacing $\mathcal{L}_X g$ with its initial expression $\mathcal{Z}_0g=-2\lambda\|\xi\|^2g$ :

$$\mathcal{L}_X \left( \mathcal{Z}_mg \right) = \left( X(\mathcal{Z}_m) \right) g + \mathcal{Z}_0\mathcal{Z}_m g.$$ Since $\mathcal{Z}_{m+1}=\left( X(\mathcal{Z}_m) \right) + \mathcal{Z}_0\mathcal{Z}_m $, we have $\mathcal{L}^{m+1}_X(-2\lambda\|\xi\|^2g)=\mathcal{Z}_{m+1}g$.
We thus obtain
\begin{align*}
\mathcal{L}_X^{m+1}
\left(
\lambda(2-n)\,\xi^\flat\otimes\xi^\flat
\right)
+\mathcal{L}^{m+1}_X(-2\lambda|\xi|^2g)={}&
(2-n)
\sum_{j=0}^{m+1}
\sum_{i=0}^{j}
\binom{m+1}{j}\binom{j}{i}\\
&\qquad\times
X^{m+1-j}(\lambda)
\left(
\mathcal{L}_X^{j-i}\xi^\flat
\otimes
\mathcal{L}_X^i\xi^\flat
\right)+\mathcal{Z}_{m+1}g.
\end{align*}
This is precisely formula \eqref{1} at rank $m+1$, with:

$$\mathcal{Z}_{m+1} = X(\mathcal{Z}_{m}) + \mathcal{Z}_{0}\mathcal{Z}_{m} \quad \text{and} \quad \mathcal{Z}_0=-2\lambda\|\xi\|^2\quad \text{for all}\quad m\geq 1.$$
\end{enumerate}

\end{proof}
\begin{prop} 
\label{PP2}
Let $(M,g)$, with $n\geq2$, be a Riemannian manifold satisfying relation \eqref{eq:ricci} such that the quadruple $(M,g,X,\mu)$ is a Ricci solution. The vector field $X$ is a symmetry of the $k$-th order tensor $\mathcal{L}_X^k g$ (for $k\geq 1$) if and only if \begin{equation}\label{HOMOGENE}
    \left(\Phi_t^* \mathcal{L}_X^k g \right)_p=e^{2\mu t}\Big(\mathcal{L}^k_X g\Big)_p
\end{equation}  
for all $p\in M$ and all $t$ for which the flow is defined.
\end{prop}
\begin{proof}
Suppose that the quadruple $(M,g,X,\mu)$ is a Ricci solution.

    Fix $p\in M$ and set $ G_p(t):=(\Phi_t^*g)_p.$
    
    For each $t$, $G_p(t)$ belongs to $\operatorname{Sym}^2(T_p^*M).$

By \eqref{eq:derivees-successives},
\begin{equation}
    G_p^{(k)}(t)  =  \left( \Phi_t^*\mathcal{L}_X^kg
    \right)_p.
    \label{eq:G-derivee}
\end{equation}
Suppose that $\mathcal{L}^k_X\operatorname{Ric}=0$; then the Ricci soliton equation becomes \begin{equation}\label{A}
\mathcal{L}^k_X\operatorname{Ric}+\frac{1}{2}\mathcal{L}^{k+1}_Xg=\mu \mathcal{L}^k_Xg
 \end{equation} which yields $$\mathcal{L}^{k+1}_Xg=2\mu \mathcal{L}^k_Xg$$
Taking the pullback by $\Phi_t$, we obtain $\Phi^*_t(\mathcal{L}_X^{k+1}g)=2\mu\Phi_t^* (\mathcal{L}_X^kg).$

By \eqref{eq:G-derivee}, $  G_p^{(k+1)}(t) = 2\mu G_p^{(k)}(t)$, so 
\begin{equation}
    G_p^{(k+1)}(t) -  2\mu G_p^{(k)}(t)  = 0
    \label{eq:ode-Gp}
\end{equation} 
Now let $  H_p(t):= G_p^{(k)}(t).$ Then $ H_p'(t)=G_p^{(k+1)}(t).$

Equation \eqref{eq:ode-Gp} becomes
\begin{equation}
    H_p'(t)-2\mu H_p(t)=0.
    \label{eq:ode-H}
\end{equation}
The solution to this linear homogeneous differential equation is 
   $$ H_p(t)=e^{2\mu t}H_p(0)=e^{2\mu t}G_p^{(k)}(0)=e^{2\mu t}\Big(\mathcal{L}^k_Xg\Big)_p$$

Conversely, suppose that \eqref{HOMOGENE} holds for all $p$ and all $t$. Since \eqref{HOMOGENE} is the solution to \eqref{eq:ode-Gp}, at $t=0$ given that $\Phi_0=\operatorname{Id}$,  we have $G_p^{(k+1)}(0)=2\mu G_p^{(k)}(0)$. However, $$G_p^{(k+1)}(0) =(\mathcal{L}_X^{k+1}g)_p\qquad\text{and}\qquad G_p^{(k)}(0) =(\mathcal{L}_X^kg)_p.$$

Thus, $(\mathcal{L}_X^{k+1}g)_p=2\mu(\mathcal{L}_X^kg)_p.$

Since $p$ is arbitrary, $\mathcal{L}_X^{k+1}g=2\mu \mathcal{L}_X^{n-1}g.$

Substituting this into equation \eqref{A} yields the result.
\end{proof}
\begin{thm}\label{thm:principal}
Let $(M,g)$ be a Riemannian manifold of dimension $n \geq 2$ with a curvature tensor satisfying relation \eqref{eq1}, such that $(M,g,X,\mu)$ is a Ricci soliton.
The following assertions are equivalent:

\begin{enumerate}[label=\textnormal{(\roman*)}]

    \item
    \begin{equation}
      (2-n)\sum_{j=0}^{k}\sum_{i=0}^{j}\binom{k}{j}\binom{j}{i}
X^{k-j}(\lambda)\times
\left(
\mathcal{L}_X^{j-i}\xi^\flat
\otimes
\mathcal{L}_X^i\xi^\flat
\right)=- \mathcal{Z}_kg
   \end{equation} where  $X^{k-j}(\lambda)=\mathcal{L}^{k-j}_X\lambda$, $\mathcal{Z}_k=X(\mathcal{Z}_{k-1})+\mathcal{Z}_0\mathcal{Z}_{k-1}$ for all $k\ge 1$ and $\mathcal{Z}_0=-2\lambda |\xi|^2$;
   \item the vector field $X$ is a symmetry of the Ricci tensor of order $k \ge 1$;
    \item for all  $p\in M$,
    \begin{equation}
    \frac{\mathrm{d}^{k+1}}{\mathrm{d}t^{k+1}} (\Phi_t^*g)_p - 2\mu \frac{\mathrm{d}^k}{\mathrm{d}t^k} (\Phi_t^*g)_p = 0
    \label{eq:principal-ii}
\end{equation}

    \item for all $p\in M$ and all $u,v\in T_pM$,
    \begin{equation}
   \frac{\mathrm{d}^{k+1}}{\mathrm{d}t^{k+1}}
        g_{\gamma_p(t)}
        \left(
        (\mathrm{d}\Phi_t)_pu,
        (\mathrm{d}\Phi_t)_pv
        \right)
        - 2\mu \frac{\mathrm{d}^k}{\mathrm{d}t^k}
        g_{\gamma_p(t)}
        \left(
        (\mathrm{d}\Phi_t)_pu,
        (\mathrm{d}\Phi_t)_pv
        \right)
        =0
    \label{eq:principal-iii}
    \end{equation}

\end{enumerate}

Moreover, if the Ricci soliton is stationary, then
\begin{equation}
    \Phi_t^*g = \sum_{j=0}^{k} \frac{t^j}{j!}
    \mathcal{L}_X^jg.
    \label{eq:principal-c-zero}
\end{equation}

if the Ricci soliton is not stationary, then
\begin{equation}
\Phi_t^*g
    =
    \sum_{j=0}^{k-1}
    \frac{t^j}{j!}
    \left(
    \mathcal{L}_X^jg-(2\mu)^{j-k}\mathcal{L}_X^kg.
    \right)
    +\frac{e^{2\mu t}}{2^k\mu^k}
    \mathcal{L}_X^kg.
    \label{B}
\end{equation}
\end{thm}
\begin{proof}
   The equivalence between \textnormal{(i)} and \textnormal{(ii)} is given by Proposition \ref{PP1}.

The equivalence between \textnormal{(ii)} and \textnormal{(iii)}

is exactly the content of Proposition \ref{PP2}.

For the equivalence between \textnormal{(iii)} and
\textnormal{(iv)}, we use  $$ (\Phi_t^*g)_p(u,v) =  g_{\Phi_t(p)} \left( (d\Phi_t)_pu, (d\Phi_t)_pv \right).$$
Since $\gamma_p(t)=\Phi_t(p),$ we obtain \eqref{eq:principal-iii}.
Consequently: \begin{itemize}[label=$\bullet$]
    \item If the Ricci soliton is stationary, then $\mu=0$, which yields $\mathcal{L}_X^{k+1}g=0$; therefore, \begin{equation}
    \frac{d^{k+1}}{dt^{k+1}}\Phi_t^*g=0.
    \label{m1}
\end{equation} Equation \eqref{m1} directly implies that the tensor-valued function $t \mapsto \Phi_t^*g$ has a vanishing $(k+1)$-th derivative.

It is therefore a polynomial in $t$ of degree at most $k$.

Taylor's formula is then:
$$ \Phi_t^*g  = \sum_{j=0}^{k}  \frac{t^j}{j!} \left.\frac{d^j}{dt^j}\right|_{t=0}\varphi_t^*g.$$

According to \eqref{eq:derivees-successives},

$ \left. \frac{d^j}{dt^j} \right|_{t=0} \Phi_t^*g = \mathcal{L}_X^jg.$

Therefore, we obtain $$ \Phi_t^*g=\sum_{j=0}^{k}
    \frac{t^j}{j!} \mathcal{L}_X^jg.$$

\item if the Ricci soliton is not stationary, then

Let $G(t)=\Phi_t^*g.$

The differential equation derived from \eqref{A} is $ G^{(k+1)} -2\mu G^{(k)} = 0.$

Its characteristic polynomial is $  r^{k+1}-2\mu r^k=  r^k(r-2\mu).$

The root $0$ has multiplicity $k$ and the root $2\mu$ is simple.

The general solution is therefore of the form
\begin{equation}
    G(t) = \sum_{j=0}^{k-1}
    \frac{t^j}{j!}A_j
    +
    e^{2\mu t}B,
    \label{eq:solution-Ak}
\end{equation}
where $A_j$ and $B$ are symmetric tensors independent of $t$.

For $j \le k-1$, let us differentiate \eqref{eq:solution-Ak} $j$ times
and evaluate at $t=0$.

We obtain $G^{(j)}(0) = A_j + (2\mu)^j B$.

But $G^{(j)}(0) = \mathcal{L}_X^j g$.

Consequently,
\begin{equation}
    A_j  =\mathcal{L}_X^jg-(2\mu)^jB.
    \label{eq:def-Ak}
\end{equation}

For $j=k$, the polynomial part vanishes after differentiation,
and we obtain $G^{(k)}(0) =(2\mu)^kB.$ However, $G^{(k)}(0)  =\mathcal{L}_X^kg.$

Thus \begin{equation}\label{BB}
    B=\frac{1}{2^k\mu^k}
    \mathcal{L}_X^kg.
\end{equation}

By substituting \eqref{eq:def-Ak} and \eqref{BB} into \eqref{eq:solution-Ak},
we obtain \eqref{B}.
\end{itemize}
\end{proof}
\begin{prop}
    Let $p\in M$ and let $\gamma_p(t)=\Phi_t(p)$ be the integral curve of a smooth vector field $X$ passing through $p$ such that $(M,g,X,\mu)$ is a Ricci soliton.

Let $u,v\in T_pM$.
Define
\begin{equation}
    h_{p,u,v}(t)  := (\Phi_t^*g)_p(u,v).
    \label{h1}
\end{equation}
$X$ is a Ricci tensor symmetry of order $k\ge 1$ if and only if
\begin{equation}
    h_{p,u,v}^{(k+1)}(t)  -  2\mu h_{p,u,v}^{(k)}(t)= 0.
    \label{eq:ode-h}
\end{equation}
\end{prop}
\begin{proof}
Suppose that $\mathcal{L}_X^k\operatorname{Ric}=0$ and $(M,g,X,\mu)$ is a Ricci soliton.

    By the definition of the pullback of equality \eqref{h1}, we find:
\begin{equation}
    h_{p,u,v}(t)= g_{\gamma_p(t)}\left((d\varphi_t)_pu,(d\varphi_t)_pv\right).
    \label{eq:h-geometrique}
\end{equation}

Thus, $h_{p,u,v}(t)$ represents the inner product—with respect to the metric at the point $\gamma_p(t)$—of two vectors transported by the flow.

According to relation \eqref{h1}, $h_{p,u,v}(t) = G_p(t)(u,v)$.

The condition $\mathcal{L}_X^k\operatorname{Ric}=0$ is equivalent to Proposition \ref{PP2}. Furthermore, since the vectors $u$ and $v$ are fixed in $T_pM$, we have $h_{p,u,v}^{(k)}(t) = G_p^{(k)}(t)(u,v)$.

Evaluating equation \eqref{eq:ode-Gp} on $(u,v)$ yields the result.
\end{proof}

\subsection{Finite-order Lie fields}
In this subsection, we assume that the vector field $X$ we use is a conformal vector field and that its infinitesimal flow preserves the line distribution $\mathcal{D} = \operatorname{Span}\{\xi\}$; that is, there exist a smooth function $\varphi$ and a real number $a$ such that $\mathcal{L}_Xg = 2\varphi g$ and $\mathcal{L}_X\xi = a\xi$, respectively.
\begin{ex}
    Consider the Riemannian manifold given in Example \ref{exemple1} and set $\mathfrak{f}(t) = t^k$ with $k \neq 0$, where $I \subset \mathbb{R}$ is an open interval.

Consider the vector field $X$ defined on $M$ by:
\begin{equation}
\label{eq:X_power}
X = -a t\,\partial_t + a(k - 1) x\,\partial_x - a \sum_{j=1}^{n-2} y_j\partial_{y_j},
\end{equation}
where $a$ is a real constant.
Evaluating the Lie bracket by its action on a smooth test function $f$, we obtain:
\begin{align*}
\left(\mathcal{L}_X\xi\right)f=[X,\xi]f=[X, \partial_t]f &= X(\partial_t f) - \partial_t(X f)= \left( -at\,\partial_t + a(k-1)x\,\partial_x - a\sum_{j=1}^{n-2} y_j\partial_{y_j} \right)(\partial_t f) \\
&- \partial_t\left( -at\,\partial_t f + a(k-1)x\,\partial_x f - a\sum_{j=1}^{n-2} y_j\partial_{y_j} f \right) \\
&= -at\,\partial_t^2 f + a(k-1)x\partial_x\partial_t f - a\sum_{j=1}^{n-2} y_j\partial_{y_j}\partial_t f \\
& - \left( -a\partial_t f - at\partial_t^2 f + a(k-1)x\,\partial_t\partial_x f - a\sum_{j=1}^{n-2} y_j\partial_t\partial_{y_j} f \right) = a\partial_t f=a\xi(f).
\end{align*} So $$\mathcal{L}_X\xi=a\xi.$$
We also have: $$(\mathcal{L}_X g)_{tt} = 2\,\partial_t(-at) = -2a = -2ag_{tt}.$$
\begin{align*}
(\mathcal{L}_X g)_{xx} &= -at\partial_t(t^{2k}) + 2\,t^{2k}\,\partial_x\Big(a(k-1)x\Big) = (-at)(2k t^{2k-1}) + 2a(k-1)t^{2k} \\
&= -2ak t^{2k} + (2ak - 2a)t^{2k}= -2at^{2k} = -2ag_{xx}.
\end{align*}
With  $X_j = -a y_j$ :
$$(\mathcal{L}_X g)_{y_\alpha y_i} = \partial_{y_i} X_j + \partial_{y_j} X_\beta = -a\delta_{ji} - a\delta_{ji} = -2a\delta_{j i} = -2ag_{y_j y_i}.$$
All the cross derivatives cancel out identically: $$(\mathcal{L}_X g)_{tx} = (\mathcal{L}_X g)_{t y_j} = (\mathcal{L}_X g)_{x y_j} = 0.$$ By grouping all these contributions, we obtain  $\mathcal{L}_X g = -2a g$.
\end{ex}
\begin{lem}
Under these given conditions, 
we have: $\mathcal{L}_X\xi^\flat=(a+2\varphi)\xi^\flat.$
\end{lem}
\begin{proof}
By definition, $\xi^\flat=g(\xi,\cdot).$
So $\mathcal{L}_X\xi^\flat
=
(\mathcal{L}_Xg)(\xi,\cdot)
+
g(\mathcal{L}_X\xi,\cdot).$
The hypotheses give $(\mathcal{L}_Xg)(\xi,\cdot)
=
2\varphi\xi^\flat\qquad\text{et}\qquad g(\mathcal{L}_X\xi,\cdot)
=
a\xi^\flat.$ Hence $\mathcal{L}_X\xi^\flat
=
(a+2\varphi)\xi^\flat.$
\end{proof}
\begin{rem}
    It follows that \begin{equation}
\mathcal{L}_X\eta
=
2(a+2\varphi)\eta.
\label{and}
\end{equation} 
By reporting in \eqref{R}, we obtain
\begin{align}
\mathcal{L}_XR
&=
X(\lambda)\eta\owedge g
+
2\lambda(a+2\varphi)\eta\owedge g
+
2\varphi\lambda\eta\owedge g
\nonumber\\
&=
\left[
X(\lambda)
+
2a\lambda
+
6\varphi\lambda
\right]
\eta\owedge g.
\label{eq:LXR_scalar}
\end{align}
\end{rem}
Posons maintenant  $q_0=\lambda,$ et définissons récursivement
\begin{equation}
q_{m+1}=X(q_m)+(2a+6\varphi)q_m.
\label{eq:recurrence}
\end{equation}

Alors nous avons le résultat suivant.
\begin{prop}
Pour tout entier $m\geq0$, $\mathcal{L}_X^mR
=q_m\eta\owedge g,$
où $q_{m+1}=X(q_m)+(2a+6\varphi)q_k.$
\end{prop}

\begin{proof}
For $m=0$, the identity is precisely $R=q_0\eta\owedge g.$

Suppose $\mathcal{L}_X^mR=q_m\eta\owedge g.$

Then $\mathcal{L}_X^{m+1}R=\mathcal{L}_X(q_m\eta\owedge g).$

By Leibniz, $$\mathcal{L}_X^{m+1}R=X(q_m)\eta\owedge g+q_m(\mathcal{L}_X\eta)\owedge g+q_m\eta\owedge\mathcal{L}_Xg.$$
The relation \eqref{eq:recurrence} and the fact that $X$ is conformal give us: $\mathcal{L}_X^{m+1}R
=\left[
X(q_m)+2(a+2\varphi)q_m+2\varphi q_m
\right]\eta\owedge g.$

Thus $\mathcal{L}_X^{m+1}R=q_{m+1}\eta\owedge g,$ with $q_{m+1}=X(q_m)+(2a+6\varphi)q_m.$
\end{proof}
\begin{rem}
  Under the given hypotheses,
all iterates of $R$ belong to $\operatorname{Span}_{C^\infty(M)}
\{\eta\owedge g\}$ and the space of iterates $\mathcal{V}_X(R)==
\operatorname{Span}_{\mathbb R}
\left\{
R,\mathcal{L}_XR,\mathcal{L}_X^2R,...,\mathcal{L}^mR,...
\right\}$ has rank $1$.
\end{rem}
\begin{prop}
    Let $\alpha=2a+6\varphi$ and define the operator $D_X=X+\alpha$ with $D^0_X=\lambda$. Then $$q_m=D_X^m(\lambda).$$ 
\end{prop}
\begin{proof}
We see that for $m= 0$, $D_X^0(\lambda) = \lambda = q_0$

Let us assume that $q_m=D_X^m(\lambda)$ holds for some rank $m \ge 0$.

We show that $q_{m+1}=D_X^{m+1}(\lambda).$

We have: 
$$D_X^{m+1}(\lambda)=D_X\Big(D_X^{m}(\lambda)\Big)=D_X(q_m) = X(q_m) + \alpha q_m = X(q_m) + (2a+6\varphi)q_m$$
However, according to the recurrence relation in the statement, we have precisely:
$$q_{m+1} = X(q_m) + (2a+6\varphi)q_m$$
Thus, the equality holds.
\end{proof}
\begin{rem}\label{re}
    For a fixed integer $k\geq0$, the vector field $X$ is a Lie symmetry of order $k+1$ if and only if $D_X^{k}(\lambda)\ne0$ and $D_X^{k+1}(\lambda)=0.$
\end{rem}
\begin{thm}
Remark \ref{re} shows that the geometric problem of determining the minimal order
of $R$ with respect to $X$ is equivalent to the scalar differential
problem of determining the minimal order of $\lambda$ with respect to the operator $D_X$.
\end{thm}
\begin{proof}
    Given that $\mathcal{L}^mR=q_m\eta\owedge g$ and $D^m_X(\lambda)=q_m$, it follows that $q_m = 0 \iff D_X^m(\lambda) = 0$.

Consequently, the set of indices $m \ge 0$ for which $\mathcal{L}_X^m R = 0$ coincides exactly with the set of indices $m \ge 0$ for which $D_X^m(\lambda) = 0$. By taking the smallest element of these sets (the minimum), we conclude that the two minimal orders coincide:
$$\min \left\{ m \ge 1 : \mathcal{L}_X^m R = 0 \right\} = \min \left\{ k \ge 1 : D_X^m \lambda = 0 \right\}.$$
\end{proof}
\begin{ex}
    Consider a manifold locally equipped with coordinates $(t,x_2,\ldots,x_n)$ and assume that $X=\frac{\partial}{\partial t}.$
Assume also that $\varphi=-\frac{a}{3}$; we then have $$D_X=X \quad \text{and }\quad  \mathcal{L}_X^{m+1}R=\frac{\partial^{m+1}\lambda(t,x_2,.....,x_n)}{\partial t^{m+1}}
\eta\owedge g.$$

For a non-negative integer $k\ge 0$, let $\lambda(t,x_2,.....,x_n)=\sum\limits_{j=0}^{k}a_j(x_2,....,x_n)t^j$ with the condition $a_k(x_2,.....,x_n)\ne 0$. Then $\frac{\partial^{k+1}\lambda(t,x_2,.....,x_n)}{\partial t^{k+1}}=0$ but $\frac{\partial^k\lambda(t,x_2,.....,x_n)}{\partial t^{k}}
=k!a_k(x_2,.....,x_n)\ne0.$

We thus have $\mathcal{L}_X^{k+1}R=0,
\qquad
\mathcal{L}_X^kR\neq0.
$
\end{ex}
\begin{prop}
Let $\gamma(t)$ be an integral curve of the vector field $X$, and let $F_m(t) = q_m(\gamma(t))$, with the initial condition $F_0(t) = \lambda(\gamma(t))$. Then the sequence of functions $(F_m)_{m \in \mathbb{N}}$ satisfies the ordinary differential equation:
\begin{equation}\label{suite}
    F_{m+1}(t) = F_m'(t) + \alpha(\gamma(t)) \, F_m(t).
\end{equation}
\end{prop}

\begin{proof}
By the definition of an integral curve, we have $\dot{\gamma}(t) = X_{\gamma(t)}$. Evaluating the recurrence relation at a point $\gamma(t)$ yields:
$$q_{m+1}(\gamma(t)) = (X(q_m))(\gamma(t)) + \alpha(\gamma(t))  q_m(\gamma(t))$$
Since the action of the field $X$ on a function along its integral curve corresponds to the ordinary derivative with respect to the parameter $t$, we have $(X q_m)(\gamma(t)) = \frac{\mathrm{d}}{\mathrm{d}t}\left(q_m(\gamma(t))\right) = F_m'(t)$. Setting $F_m(t) = q_m(\gamma(t))$ finally yields the result. 
\end{proof}

\begin{prop}
The general solution to the recurrence relation \eqref{suite} is given for all $m \in \mathbb{N}$ by:

$$F_m(t) =  \exp\left( -\int_0^t \alpha(\gamma(s)) \, \mathrm{d}s \right) \frac{\mathrm{d}^m}{\mathrm{d}t^m} \left(  \exp\left( \int_0^t \alpha(\gamma(s)) \, \mathrm{d}s \right) \lambda(\gamma(t)) \right).$$
\end{prop}

\begin{proof}
Let us let $I(t)= \exp\left(\int_0^t \alpha(\gamma(s)) \, \mathrm{d}s \right)$ such that $I'(t) = \alpha(\gamma(t))I(t)$, and let us set a new variable $v_k(t) = I(t)F_m(t)$, i.e. $F_m(t) = \frac{v_m(t)}{I(t)}$.

By injecting this expression into the recurrence relation for rank $m+1$ given by \eqref{suite} we obtain:

$$v_{m+1}(t) = I(t) \left[ \left(\frac{v_m(t)}{I(t)}\right)' + \alpha (\gamma(t))\frac{v_m(t)}{I(t)} \right]$$

By developing the derivative of the quotient:

$$\left(\frac{v_m(t)}{I(t)}\right)' = \frac{v_m'(t) I(t) - v_m(t) I'(t)}{I^2(t)} = \frac{v_m'(t)}{I(t)} - \alpha(\gamma(t)) \frac{v_m(t)}{I(t)}.$$

It then comes:

$$v_{m+1}(t) = I(t) \left( \frac{v_m'(t)}{I(t)} - \alpha(\gamma(t)) \frac{v_m(t)}{I(t)} + \alpha(\gamma(t)) \frac{v_m(t)}{I(t)} \right) = v_m'(t)$$

The recurrence is thus simplified into a simple successive derivation:

$$v_m(t) = \frac{\mathrm{d}^m}{\mathrm{d}t^m} v_0(t)$$

Since $v_0(t) = I(t)F_0(t) = I(t)\lambda(\gamma(t))$, we obtain $v_m(t) = \frac{\mathrm{d}^m}{\mathrm{d}t^m}\left( I(t)\lambda(t) \right)$. Returning to $F_m(t) = \frac{v_m(t)}{I(t)}$, we conclude:

$$F_m(t) = \frac{1}{I(t)} \frac{\mathrm{d}^m}{\mathrm{d}t^m} \left( I(t) \lambda(\gamma(t)) \right).$$ This yields the result.
\end{proof}

\begin{thm}
Let $\gamma(t)$ be an integral curve of a vector field $X$ such that $F_m(t) \neq 0$ and $F_{m+1}(t) = 0$ for some integer $m\ge 0$. Then, the function $\lambda$ along the curve $\gamma(t)$ is given by:
$$\lambda(\gamma(t)) = \exp\left(-\int_0^t \alpha(\gamma(s))\,\mathrm{d}s\right)\sum_{j=0}^mp_jt^j$$
where the $p_j$ are real constants such that $p_k$ is non-zero.
\end{thm}
\begin{proof}
The condition $F_{m+1}(t) = 0$ translates to

$$ \exp\left( -\int_0^t \alpha(\gamma(s)) \, \mathrm{d}s \right) \frac{\mathrm{d}^{m+1}}{\mathrm{d}t^{m+1}} \left(  \exp\left( \int_0^t \alpha(\gamma(s)) \, \mathrm{d}s \right) \lambda(\gamma(t)) \right)=0$$
so $$ \frac{\mathrm{d}}{\mathrm{d}t}\left(\frac{\mathrm{d}^{m}}{\mathrm{d}t^{m}} \left(  \exp\left( \int_0^t \alpha(\gamma(s)) \, \mathrm{d}s \right) \lambda(\gamma(t)) \right)\right)=0$$
Since $F_k(t) \neq 0$, there exists a non-zero real constant $C$ such that $$\frac{\mathrm{d}^{m}}{\mathrm{d}t^m} \left(  \exp\left( \int_0^t \alpha(\gamma(s))  \mathrm{d}s \right) \lambda(\gamma(t)) \right)=C.$$
This yields  $$ \exp\left( \int_0^t \alpha(\gamma(s)) \, \mathrm{d}s \right) \lambda(\gamma(t))=Ct^m+\sum_{j=0}^{m-1}p_jt^j$$  yielding the result.
\end{proof}
\begin{rem}
The dynamical condition $\mathcal{L}_X^kR = R$—for a fixed natural number $k$ along the infinitesimal flow preserving the line distribution $\mathcal{D} =\operatorname{Span}\{\xi\}$ of a conformal vector field $X$ imposes a rigid compatibility between the algebraic structure of $R$, the scalar factor $\lambda$, the 1-form $\xi^\flat$, and the Riemannian geometry of $M$.

\begin{itemize}[label=$\bullet$]
    \item The operator equation $(\mathcal{L}_X^k - \text{Id})R = 0$ implies that the components of $R$ lie in the kernel of a polynomial whose characteristic roots satisfy $r^k = 1$.
    \begin{itemize}[label=$\star$]
        \item If $k = 1$ ($\mathcal{L}_X R = R$), the curvature tensor undergoes pure exponential dilation along the flow orbits: $\Phi_t^* R = e^t R$.
        \item If $k$ is even, decay or oscillatory behaviors ($r = -1$) are possible, whereas if $k$ is odd, only exponential growth ($r= 1$) occurs.
    \end{itemize}
    
    \item Due to the positivity and non-degeneracy of the Riemannian metric $g$, this recurrence condition strongly constrains the geometry of the manifold; it extends via contraction to the derived tensors $\text{Ric}$ and the scalar curvature, and yields the relation $D^k_X(\lambda)=\lambda.$
\end{itemize}
\end{rem}
\begin{prop} 
\begin{enumerate}
    \item $\mathcal{L}^2_XR=R$ if and only if
\begin{equation}
X^2(\lambda)+2\alpha X(\lambda)+\Big(X(\alpha)+\alpha^2-1\Big)\lambda=0.
\label{d2}
\end{equation}
\item $\mathcal{L}^3_XR=R$ if and only if \begin{equation}
X^3(\lambda)+3\alpha X^2(\lambda)+\Big(3\alpha^2 +X(\alpha)\Big)X(\lambda)+\Big(X^2(\alpha)+3\alpha X(\alpha)+\alpha^3-1\Big)\lambda=0.
\label{d3}
\end{equation}
\end{enumerate}
\end{prop}
\begin{proof}
   We have  $\mathcal{L}^k_XR=R\iff D_X^k(\lambda)=\lambda$  so
    \begin{enumerate}
        \item $\mathcal{L}^2_XR=R\iff D_X^2(\lambda)=\lambda.$ $D_X^2(\lambda)=D_X(D_X(\lambda)).$
Since $D_X(\lambda)=X(\lambda)+\alpha \lambda$, we obtain  $$D_X^2(\lambda)=(X+\alpha)(X(\lambda)+\alpha\lambda)=X(X(\lambda))+X(\alpha\lambda)+\alpha X(\lambda)+\alpha^2\lambda$$
while $X(\alpha\lambda)=X(\alpha)\lambda+\alpha X(\lambda).$

Consequently, $$D_X^2(\lambda)
=X^2(\lambda)+2\alpha X(\lambda)+\Big(X(\alpha)+\alpha^2\bigr)\lambda.$$
The equation $D_X^2(\lambda)=\lambda$ then yields $X^2(\lambda)+2\alpha X(\lambda)+\Big(X(\alpha)+\alpha^2-1\Big)\lambda=0.$
\item $\mathcal{L}^3_XR=R\iff D^3_X(\lambda)=\lambda$. For this reason $$\lambda=D_X\Big(D^2_X(\lambda)\Big)=D_X\left(X^2(\lambda)+2\alpha X(\lambda)+\Big(X(\alpha)+\alpha^2\bigr)\lambda\right)$$ so $$\lambda=(X+\alpha)\Big(X^2(\lambda)+2\alpha X(\lambda)+\Big(X(\alpha)+\alpha^2\Big)\lambda\Big)$$ which results in \begin{align*}
    \lambda&=X^3(\lambda)+2X(\alpha)X(\lambda)+2\alpha X^2(\lambda)+\lambda X^2(\alpha)+X(\alpha)X(\lambda)+\alpha^2X(\lambda)+X(\alpha^2)\lambda\\
    &+\alpha X^2(\lambda)+2\alpha^2 X(\lambda)+\Big(\alpha X(\alpha)+\alpha^3\Big)\lambda
\end{align*} while $X(\alpha^2)=2\alpha X(\alpha)$ so $$\lambda=X^3(\lambda)+3\alpha X^2(\lambda)+\Big(3\alpha^2 +X(\alpha)\Big)X(\lambda)+\Big(X^2(\alpha)+3\alpha X(\alpha)+\alpha^3\Big)\lambda$$
    \end{enumerate}
\end{proof}
\begin{rem}
    The differential equations \eqref{d2} and \eqref{d3} clearly show that the direct expansion successively yields the derivatives $$X(\lambda),\quad X^2(\lambda),\quad X^3(\lambda),....$$
    For an arbitrary order, it is therefore preferable to abandon this approach and instead directly exploit the structure of the flow of $X$.
\end{rem}
Let $\Phi_t$ denote the local flow of $X$. For a point $p \in M$, the integral curve passing through $p$ is given by $\gamma_p(t)=\Phi_t(p),$
and satisfies
\begin{equation}
\gamma_p'(t)=X_{\gamma_p(t)}.
\label{eq:gamma}
\end{equation}

\begin{lem}\label{l1}
Let $f\in C^\infty(M)$ and $h_p(t)=f(\Phi_t(p))$. SO
\begin{equation}
h_p'(t)=X(f)(\Phi_t(p)).
\label{eq:derivative-flow}
\end{equation}
\end{lem}

\begin{proof}
By the chain derivative rule,
$h_p'(t)=\mathrm{d}f_{\Phi_t(p)}\left(\frac{\mathrm{d}}{\mathrm{d} t}\Phi_t(p)
\right).$

Like $\frac{\mathrm{d}}{\mathrm{d} t}\Phi_t(p)=X_{\Phi_t(p)},$
we obtain $$h_p'(t)
=\mathrm{d} f_{\Phi_t(p)}
(X_{\Phi_t(p)})=
X(f)(\Phi_t(p)).$$
\end{proof}
Without loss of generality let us put \begin{equation}
\alpha_p(t)
=\alpha(\Phi_t(p))=6
\varphi(\Phi_t(p))+2a.
\label{eq:qp}
\end{equation} 
\begin{prop}
For any function $f\in C^\infty(M)$,
\begin{equation}
\Big(D_X (f)\Big)(\Phi_t(p))
=h_p'(t)+\alpha_p(t)h_p(t).
\label{eq:P-orbite}
\end{equation}
\end{prop}

\begin{proof}
By definition, $D_X(f)=X(f)+\alpha f.$

Evaluating at the point $\Phi_t(p)$, $(D_X(f))(\Phi_t(p))=X(f)(\Phi_t(p))+\alpha(\Phi_t(p))f(\Phi_t(p)).$
By lemma \ref{l1}, $X(f)(\Phi_t(p))=h_p'(t).$ Thus $(D_X(f))(\Phi_t(p))
=
h_p'(t)+\alpha_p(t)h_p(t).$
\end{proof}
\begin{thm}\label{th1}
  For a fixed non-negative integer $k$, the relation $\mathcal{L}^k_X R=R$ along the orbit is equivalent to the ordinary differential equation
\begin{equation}
\left(
\frac{\mathrm{d}}{\mathrm{d}t}+\alpha_p(t)
\right)^k
u_p(t)
=u_p(t),
\label{eq:ODE-orbit} 
\end{equation} where $u_p(t)=\lambda(\Phi_t(p))$
\end{thm}
\begin{proof} Let us prove by induction that for every integer $m \ge 1$, we have: $$\Big(D^m_X (f)\Big)(\Phi_t(p))=\left(
\frac{\mathrm{d}}{\mathrm{d}t}+\alpha_p(t)
\right)^m
h_p(t)
$$
For $m=0$, the statement $f(\Phi_t(p))=h_p(t)$ holds, and for $m=1$, the result is exactly \eqref{eq:P-orbite}. Suppose that $$D_X^m(f(\Phi_t(p)))=\left(
\frac{\mathrm{d}}{\mathrm{d} t}+\alpha_p(t)
\right)^m h_p(t).$$
Applying $D_X$ again and using
\eqref{eq:P-orbite} on the function $D_X^m(\lambda)$, we obtain
$$D_X^{m+1}(f)(\Phi_t(p))
=\left(
\frac{\mathrm{d}}{\mathrm{d} t}+\alpha_p(t)
\right)
\left[
\left(
\frac{\mathrm{d}}{\mathrm{d} t}+\alpha_p(t)
\right)^m h_p(t)
\right].$$
Thus, $D_X^{m+1}(f(\Phi_t(p)))=\left(\frac{\mathrm{d}}{\mathrm{d} t}+\alpha_p(t)
\right)^{m+1}h_p(t).$

By induction,
$$D_X^m(f)(\Phi_t(p))=\left(
\frac{\mathrm{d}}{\mathrm{d} t}+\alpha_p(t)
\right)^m h_p(t).$$ In particular, taking the function $\lambda$, the relation $D_X^k(\lambda(\Phi_t(p)))=\lambda(\Phi_t(p))$ then yields \eqref{eq:ODE-orbit}.
\end{proof}
\begin{rem}
    The presence of $\alpha_p(t)$ in \eqref{eq:ODE-orbit} might seem to
complicate the problem. However, this coefficient can be eliminated
exactly by an integrating factor.

Let us fix an interval $\mathcal{I}$ on which the orbit $\gamma_p$ is defined and
choose $t_0\in \mathcal{I}$. Let $\beta_p(t)
=
\int_{t_0}^{t}\alpha_p(s)\mathrm{d} s$ so $\beta_p'(t)=\alpha_p(t).$
\end{rem}
\begin{lem}
For any function $f\in C^\infty(\mathcal{I})$,
\begin{equation}
\left(
\frac{\mathrm{d}}{\mathrm{d} t}+\alpha_p(t)
\right)
\left(
e^{-\beta_p(t)}f(t)
\right)
=
e^{-\beta_p(t)}f'(t).
\label{eq:facteur}
\end{equation}
\end{lem}

\begin{proof}
We have : $$\frac{\mathrm{d}}{\mathrm{d} t}
\left(
e^{-\beta_p(t)}f(t)
\right)=-\beta_p'(t)e^{-\beta_p(t)}v(t)+e^{-\beta_p(t)}f'(t).$$
Since $\beta_p'=\alpha_p$,
$$\frac{\mathrm{d}}{\mathrm{d} t}
\left(e^{-\beta_p(t)}f(t)
\right)=-\alpha_p(t)e^{-\beta_p(t)}f(t)+e^{-\beta_p(t)}f(t)'.$$
we obtain $$\left(
\frac{\mathrm{d}}{\mathrm{d} t}+\alpha_p(t)
\right)(e^{-\beta_p(t)}f(t))=e^{-\beta_p(t)}f(t)'.$$
\end{proof}
\begin{thm}\label{th2}
On any integral orbit of $X$, the operator identity holds
\begin{equation}
\frac{\mathrm{d}}{\mathrm{d} t}+\alpha_p(t)
=e^{-\beta_p(t)}
\frac{\mathrm{d}}{\mathrm{d} t}e^{\beta_p(t)}.
\label{eq:operator-conjugation}
\end{equation}
Therefore, for all $m\in\mathbb{N}^*$,
\begin{equation}
\left(
\frac{\mathrm{d}}{\mathrm{d} t}+\alpha_p(t)
\right)^m
=
e^{-\beta_p(t)}
\frac{\mathrm{d}^m}{\mathrm{d}t^m}
e^{\beta_p(t)}.
\label{eq:conjugation-power}
\end{equation}
\end{thm}
\begin{proof}
The identity \eqref{eq:operator-conjugation} means that, for any
function $u\in\mathcal{C}^\infty(\mathcal{I})$, $\left(
\frac{\mathrm{d}}{\mathrm{d} t}+\alpha_p
\right)u
=e^{-\beta_p}
\frac{\mathrm{d}}{\mathrm{d} t}
\left(
e^{\beta_p}u
\right).$

Indeed,
$$e^{-\beta_p}
\frac{\mathrm{d}}{\mathrm{d} t}
\left(e^{\beta_p}u\right)=e^{-\beta_p}
\left(\beta_p'e^{\beta_p}u+e^{\beta_p}u'\right),$$
hence $e^{-\beta}
\frac{\mathrm{d}}{\mathrm{d} t}\left(e^{\beta_p}u\right)=\alpha_pu+u'.$
 So, $\frac{\mathrm{d}}{\mathrm{d} t}+\alpha_p=e^{-\beta_p}\frac{\mathrm{d}}{\mathrm{d} t}e^{\beta_p}.$

By calculation, we have: $$\left(\frac{\mathrm{d}}{\mathrm{d} t}+\alpha_p\right)^2=e^{-\beta_p}\frac{\mathrm{d}}{\mathrm{d} t}e^{\beta_p}e^{-\beta_p}\frac{\mathrm{d}}{\mathrm{d} t}e^{\beta_p}=e^{-\beta_p}\frac{\mathrm{d}^2}{\mathrm{d}^2 t}e^{\beta_p}.$$
By composing this identity $m$ times, the intermediate factors
$e^{\beta_p}e^{-\beta_p}$ cancel out: 
$$\left(e^{-\beta_p}D_te^{\beta_p}
\right)^m=e^{-\beta_p}D_t^me^{\beta_p},$$
where $D_t=\frac{\mathrm{d}}{\mathrm{d} t}.$
We thus obtain \eqref{eq:conjugation-power}.
\end{proof}
\begin{thm}\label{TH}
Let $k$ be a fixed non-negative integer; the expression for $\lambda$ when $\mathcal{L}_X^kR=R$ along an orbit $\gamma_p$ is given by: \begin{equation}
        \lambda(\Phi_t(p))
=\exp\left(
-\int_{t_0}^{t}
[6\varphi(\Phi_s(p))+2a]\mathrm{d} s
\right)\left(C_1 e^t + \delta_k C_2 e^{-t}
+
\sum_{j=1}^{\lfloor \frac{k-1}{2} \rfloor} e^{\cos\left(\frac{2\pi j}{k}\right) t} \psi_j\right)
    \end{equation} where $\psi_j= A_j \cos\left(\sin\left(\frac{2\pi j}{k}\right) t\right) + B_j\sin\left(\sin\left(\frac{2\pi j}{k}\right)t\right)$, $\delta_k = 1$ if $k$ is even and $0$ if $k$ is odd, and $C_1, C_2, A_j, B_j \in \mathbb{R}$ are constants determined by the conditions on the orbit under consideration.
\end{thm}
\begin{proof}
By Theorem \ref{th1}, $\left(D_t+\alpha_p(t)\right)^k u_p(t)=u_p(t).$
By Theorem \ref{th2},
$$e^{\beta_p(t)}D_t^k(e^{\beta_p(t)}u_p(t))=u_p(t).$$
We obtain $D_t^k(e^{\beta_p(t)}u_p(t))
=e^{\beta_p(t)}u_p(t).$
$$\text{Let}\quad 
v_p(t)=e^{\beta_p(t)}u_p(t).\quad\text{Then}\quad  v_p^{(k)}(t)=v_p(t).$$

Conversely, if $v_p$ satisfies this equation and if $u_p(t)=e^{-\beta_p(t)}v_p(t),$

then conjugation yields $(D_t+\alpha_p(t))^ku_p(t)=u_p(t).$

The linear differential equation with constant coefficients $v^{(k)}(t) = v(t)$
is elementary to solve.

Its characteristic equation is $r^k = 1.$

Its roots are the $k$-th roots of unity $\zeta_j=e^{\frac{2\pi i j}{k}},\qquad j= 0, \ldots, k-1.$

Since the function $v(t)$ is real-valued, the general solution is obtained by grouping the real root $\zeta_0 = 1$, the root $\zeta_{\frac{k}{2}} = -1$ (when $k$ is even), and the pairs of complex conjugate roots. Depending on the parity of $k$, the real solution is written in the form:
\begin{equation}
v(t)=C_1 e^t + \delta_k C_2 e^{-t}+\sum_{j=1}^{\lfloor \frac{k-1}{2} \rfloor} e^{\mathfrak{b}_j t} \big( A_j \cos(\mathfrak{c}_j t) + B_j \sin(\mathfrak{c}_jt) \big),
\label{eq:v-general-reelle}
\end{equation}
where $\mathfrak{b}_j = \cos\left(\frac{2\pi j}{k}\right)$, $\mathfrak{c}_j = \sin\left(\frac{2\pi j}{k}\right)$, $\delta_k = 1$ if $k$ is even and $0$ if $k$ is odd, and $C_1, C_2, A_j, B_j\in \mathbb{R}$ are constants determined by the conditions on the orbit under consideration.
\end{proof}

\begin{prop}
On a domain where $X$ does not vanish, the equation $D_X^k(\lambda)=\lambda$ locally possesses a $k$-dimensional space of smooth solutions per orbit. Furthermore, for a local coordinate system $(t, x_1, \ldots, x_{n-1})$ around a point $p\in M$, the expression for the function $\lambda$ is given by: \begin{equation}
    \lambda(t,x)=e^{-\int_{t_0}^t\Big[6\varphi(s,x)+2a\Big]ds}u(t,x)
\end{equation}
with $$u(t,x)=C_0(x)e^t
+\delta_k \varepsilon_0(x)e^{-t}+
\sum_{j=1}^{\left\lfloor\frac{k-1}{2}\right\rfloor}
e^{t\cos\left(\frac{2\pi j}{k}\right)}
\Bigg[C_j(x)\cos\left(
\sin\left(\frac{2\pi j}{k}\right)t
\right)+\varepsilon_j(x)
\sin\left(
\sin\left(\frac{2\pi j}{k}\right)t
\right)
\Bigg]$$ where $\delta_k= \begin{cases} 1,&\text{if }k\text{ is even},\\ 0,&\text{if }k\text{ is old} \end{cases}$,
$C_j$ and $\varepsilon_j$ are smooth real-valued functions depending on the variables $x=(x_1,...,x_{n-1}).$
\end{prop}

\begin{proof}
Suppose that $X_p \neq 0$. Although $X$ is an arbitrary vector field on the manifold, the vector field rectification theorem guarantees the existence of a local coordinate system $(t, x_1, \ldots, x_{n-1})$ around $p$ in which the \emph{local expression} of $X$ simplifies to:
\begin{equation}
X = \frac{\partial}{\partial t}.
\label{eq:flowbox}
\end{equation}
It is precisely through this choice of local chart (where the field $X$ is "rectified") that the problem reduces to studying the system along the trajectories. In these coordinates, the operator $D_X^k$ becomes a derivative with respect to $t$ that is, \begin{equation}\label{eqdiff}
\Big(\frac{\partial}{\partial t}+\alpha\Big)^k(\lambda)=\lambda,
\end{equation} and each fixed value of $x = (x_1, \ldots, x_{n-1})$ determines a linear differential equation of order $k$ in $t$.

Let us define
\begin{equation}\label{eq:F-definition}
\beta(t,x)
=\int_{t_0}^{t} \alpha(s,x)\,ds.
\end{equation}
Then we have $\frac{\partial \beta(t,x)}{\partial t}=\alpha(t,x).$

Now let $\lambda(t,x)=e^{-\beta(t,x)}u(t,x)$,
where $u$ is a function to be determined.

We obtain
\begin{align*}
\left(
\frac{\partial}{\partial t}+\alpha(t,x)
\right)(\lambda(t,x))&=
\frac{\partial}{\partial t}
\left(e^{-\beta(t,x)}u\right)
+
\alpha(t,x) e^{-\beta(t,x)}u(t,x)\\
&=
-\beta_t(t,x) e^{-\beta(t,x)}u(t,x)
+
e^{-\beta(t,x)}u_t(t,x)
+
\alpha(t,x) e^{-\beta(t,x)}u(t,x).
\end{align*}
Since $\beta_t(t,x)=\alpha(t,x)$, the two terms containing $\alpha$ cancel each other out, and consequently, $$\left(
\frac{\partial}{\partial t}+\alpha(t,x)
\right)(\lambda(t,x))
=
e^{-\beta(t,x)}\frac{\partial u(t,x)}{\partial t}$$ thus 
\begin{equation}\label{eqd}
\left(
\frac{\partial}{\partial t}+\alpha(t,x)
\right)(\lambda(t,x))
=
e^{-\beta(t,x)}\frac{\partial \Big(e^{-\beta(t,x)}\lambda(t,x)\Big)}{\partial t}.
\end{equation}
We therefore have: \begin{align*}
    \left(
\frac{\partial}{\partial t}+\alpha(t,x)
\right)^2(\lambda(t,x))&=\left(
\frac{\partial}{\partial t}+\alpha(t,x)
\right)\left(\left(
\frac{\partial}{\partial t}+\alpha(t,x)
\right)(\lambda(t,x))\right)\\
&=\frac{\partial}{\partial t}\left(e^{-\beta(t,x)}\frac{\partial \Big(e^{-\beta(t,x)}\lambda(t,x)\Big)}{\partial t}\right)+\alpha(t,x)\left(e^{-\beta(t,x)}\frac{\partial \Big(e^{-\beta(t,x)}\lambda(t,x)\Big)}{\partial t}\right)\\
&=-\beta_t(t,x)e^{-\beta(t,x)}\frac{\partial \Big(e^{-\beta(t,x)}\lambda(t,x)\Big)}{\partial t}+e^{-\beta(t,x)}\frac{\partial^2 \Big(e^{-\beta(t,x)}\lambda(t,x)\Big)}{\partial t^2}\\
&+\alpha(t,x)e^\beta(t,x)\frac{\partial \Big(e^{-\beta(t,x)}\lambda(t,x)\Big)}{\partial t}
\\
&=-\alpha(t,x)e^{-\beta(t,x)}\frac{\partial \Big(e^{-\beta(t,x)}\lambda(t,x)\Big)}{\partial t}+e^{-\beta(t,x)}\frac{\partial^2 \Big(e^{-\beta(t,x)}\lambda(t,x)\Big)}{\partial t^2}\\
&+\alpha(t,x)e^\beta(t,x)\frac{\partial \Big(e^{-\beta(t,x)}\lambda(t,x)\Big)}{\partial t}=e^{-\beta(t,x)}\frac{\partial^2 \Big(e^{-\beta(t,x)}\lambda(t,x)\Big)}{\partial t^2}.
\end{align*}
Suppose that for every natural number $m \ge 0$, we have: $$\left(
\frac{\partial}{\partial t}+\alpha(t,x)
\right)^m(\lambda(t,x))=e^{-\beta(t,x)}\frac{\partial^m \Big(e^{-\beta(t,x)}\lambda(t,x)\Big)}{\partial t^m}$$

We therefore have:   \begin{align*}
    \left(
\frac{\partial}{\partial t}+\alpha(t,x)
\right)^{m+1}(\lambda(t,x))&=\left(
\frac{\partial}{\partial t}+\alpha(t,x)
\right)\left(\left(
\frac{\partial}{\partial t}+\alpha(t,x)
\right)^m(\lambda(t,x))\right)\\
&=\frac{\partial}{\partial t}\left(e^{-\beta(t,x)}\frac{\partial^m \Big(e^{-\beta(t,x)}\lambda(t,x)\Big)}{\partial t^m}\right)+\alpha(t,x)\left(e^{-\beta(t,x)}\frac{\partial^m \Big(e^{-\beta(t,x)}\lambda(t,x)\Big)}{\partial t^m}\right)\\
&=-\beta_t(t,x)e^{-\beta(t,x)}\frac{\partial^m \Big(e^{-\beta(t,x)}\lambda(t,x)\Big)}{\partial t^m}+e^{-\beta(t,x)}\frac{\partial^{m+1} \Big(e^{-\beta(t,x)}\lambda(t,x)\Big)}{\partial t^{m+1}}\\
&+\alpha(t,x)e^{-\beta(t,x)}\frac{\partial^m \Big(e^{-\beta(t,x)}\lambda(t,x)\Big)}{\partial t^m}
\\
&=-\alpha(t,x)e^{-\beta(t,x)}\frac{\partial^m \Big(e^{-\beta(t,x)}\lambda(t,x)\Big)}{\partial t^m}+e^{-\beta(t,x)}\frac{\partial^{m+1} \Big(e^{-\beta(t,x)}\lambda(t,x)\Big)}{\partial t{m+1}}\\
&+\alpha(t,x)e^\beta(t,x)\frac{\partial^m \Big(e^{-\beta(t,x)}\lambda(t,x)\Big)}{\partial t^m}=e^{-\beta(t,x)}\frac{\partial^{m+1} \Big(e^{-\beta(t,x)}\lambda(t,x)\Big)}{\partial t^{m+1}}.
\end{align*} In particular, for a fixed natural number $k$, we have: $$\left(
\frac{\partial}{\partial t}+\alpha(t,x)
\right)^k(\lambda(t,x))=e^{-\beta(t,x)}\frac{\partial^k \Big(e^{-\beta(t,x)}\lambda(t,x)\Big)}{\partial t^k}.$$
Thus, the equation \eqref{eqdiff} is equivalent to $$e^{-\beta(t,x)}\frac{\partial^k \Big(e^{-\beta(t,x)}\lambda(t,x)\Big)}{\partial t^k}=\lambda(t,x)$$ which directly yields $$e^{-\beta(t,x)}\frac{\partial^k \Big(u(t,x)\Big)}{\partial t^k}=e^{-\beta(t,x)}u(t,x)$$
Given that $e^{-\beta(t,x)} \neq 0$, we obtain the equation
\begin{equation}\label{eq:u-equation}
\frac{\partial^k u(t,x)}{\partial t^k}=u(t,x).
\end{equation}

We see that for each fixed $x=(x_1,\ldots,x_{n-1})$, the equation
\eqref{eq:u-equation} is a homogeneous linear differential equation
of order $k$.

The general real solution is given by:  $$u(t,x)=C_0(x)e^t
+\delta_k \varepsilon_0(x)e^{-t}+
\sum_{j=1}^{\left\lfloor\frac{k-1}{2}\right\rfloor}
e^{t\cos\left(\frac{2\pi j}{k}\right)}
\Bigg[C_j(x)\cos\left(
\sin\left(\frac{2\pi j}{k}\right)t
\right)+\varepsilon_j(x)
\sin\left(
\sin\left(\frac{2\pi j}{k}\right)t
\right)
\Bigg]$$
where $\delta_k= \begin{cases} 1,&\text{if }k\text{ is even},\\ 0,&\text{if }k\text{ is odd} \end{cases}$,
and $C_j$ and $\varepsilon_j$ are smooth real-valued functions of $x=(x_1,...,x_{n-1}).$

The $k$ constants of integration for this ODE may vary smoothly with $y$. If the transverse data are smooth, the resulting global dependence on $(t,y)$ is smooth, by the classical theory of linear differential equations with parameters.
\end{proof}
\begin{thm}
    Let $p \in M$ be a fixed point of the vector field $X$. For a fixed integer $k \geq 1$, we have $\mathcal{L}^k_X R = R$ if and only if, depending on whether $k$ is even or odd:
    \begin{itemize}[label=$\bullet$]
        \item If $k$ is even: $\varphi(p) = \frac{1-2a }{6}$ or $\varphi(p) = \frac{-1-2a }{6}$ or $\lambda(p) = 0$;
        \item If $k$ is odd: $\varphi(p) = \frac{1-2a }{6}$ or $\lambda(p) = 0$.
    \end{itemize}  
\end{thm}
\begin{proof}
$p$ being a fixed point of the vector field $X$ means that $X_p=0$, and the condition $\mathcal{L}^k_XR=R$ is equivalent to $D^k_X(\lambda)=\lambda$. Let us prove by induction that for all $m \geq 1$, $D_X^m(\lambda)(p) = \alpha(p)^m \lambda(p).$

\begin{itemize}[label=$\bullet$]
    \item For $m = 1$:

By definition, $D_X(\lambda)(p) = (X(\lambda) + \alpha \lambda)(p) = X(\lambda)(p) + \alpha(p)\lambda(p).$

Since $p$ is a fixed point of $X$ ($X_p = 0$), we have $X(\lambda)(p) = d\lambda_p(X_p) = 0$. Thus, we are left with:
$$D_X^1(\lambda)(p) = \alpha(p)^1 \lambda(p),$$
which establishes the base case for the induction.

\item Assume the property holds for a given order $m \ge 1$.

Let us show that the property holds for the order $m+1$. By the definition of an operator power:
$$D_X^{m+1}(\lambda) =D_X\left(D_X^m(\lambda) \right) = (X + \alpha)\left( D_X^m(\lambda) \right) =X\left(D_X^m(\lambda) \right)+\alpha D_X^m(\lambda).$$
Let us evaluate this expression at the point $p$:
$$D_X^{m+1}(\lambda)(p) = X\left( D_X^m(\lambda) \right)(p) + \alpha(p) \left( D_X^m(\lambda) \right)(p).$$
On one hand, since $p$ is a fixed point of $X$, applying the vector field $X$ to any function evaluated at $p$ yields zero (by the Leibniz rule or via the differential $dX_p$). Thus, for the function $ D_X^m(\lambda)$, we have: $X\left( D_X^m(\lambda) \right)(p) = 0.$
On the other hand, applying the induction hypothesis to the second term, $\left( D_X^m(\lambda) \right)(p) = \alpha(p)^m\lambda(p)$, we obtain:
$$D_X^{m+1}(\lambda)(p) = 0 + \alpha(p).\left( \alpha(p)^m \lambda(p) \right) = \alpha(p)^{m+1} \lambda(p).$$
The property is thus proved by induction for all integers $m \geq 1$. By replacing $\alpha(p)$ by $6\varphi(p) + 2a$, we obtain: $$D_X^m(\lambda)(p) = \Big(6\varphi(p) + 2a\Big)^m \lambda(p).$$

In particular $$D_X^k(\lambda)(p) = \Big(6\varphi(p) + 2a\Big)^k \lambda(p).$$
If now $\mathcal{L}_X^kR=R$ i.e. $D_X^k(\lambda) = \lambda$, then by evaluating in $p$, we have $D_X^k(\lambda(p)=\lambda(p)$, that is: $$\Big(6\varphi(p) +2 a\Big)^k \lambda(p) = \lambda(p) \iff \left(\Big(6\varphi(p) + 2a\Big)^k - 1\right)\lambda(p) = 0.$$
Therefore, $\Big(6\varphi(p) + 2a\Big)^k - 1=0$ or $\lambda(p) = 0$.
Like $6\varphi(p) + 2a\in \mathbb{R}$ so $$\forall k \in \mathbb{N}^*, \quad \mathcal{L}^k_X R = R \iff  \begin{cases} \varphi(p) = \frac{1-2a }{6}\lor \varphi(p) = \frac{-1-2a }{6} \lor \lambda(p) = 0 & \text{if } 2 \mid k \\ 
\varphi(p) = \frac{1-2a }{6} \lor \lambda(p) = 0 & \text{if } 2 \nmid k \end{cases}$$
\end{itemize}
\end{proof}
\begin{rem}
The relation $\mathcal{L}_X^m R = f \, \mathcal{L}_X^{m-1} R$, where $f$ is a fixed scalar function, expresses the fact that successive Lie derivatives of the curvature tensor are proportional from one rank to the next. Imposing this condition at a specific order $k$ simply means that the tensor $\mathcal{L}_X^{k-1} R$ is an eigenvector of the operator $\mathcal{L}_X$ associated with the eigenvalue $f$. This is a property localized at that specific rank, which by no means implies a uniform or iterative recurrence at lower orders.
\end{rem}
\begin{thm}
  \begin{enumerate}
      \item Let $k$ be a fixed non-negative integer. The function $\lambda$ defined such that for every integer $j \in \{0, \dots, k\}$, $\mathcal{L}_X^{j+1}R=f\mathcal{L}_X^jR$ (where $f$ is a function continuous on $M$ along an orbit $\gamma_p$) and the function $s \mapsto v_p(s)=\exp\left(\int_{t_0}^{s} [6\varphi(\Phi_\tau(p))+2a]\mathrm{d} \tau \right)\lambda(\Phi_s(p))$ is well-defined with $v_p^{(j)}(t_0)=c_j$ for $0 \leq j \leq k$ is unique and given by: $$\lambda(\Phi_t(p))
=\exp\left(-\int_{t_0}^{t}
[6\varphi(\Phi_s(p))+2a]\mathrm{d} s
\right)\Big(\displaystyle\sum_{j=0}^{k}
\frac{c_j}{j!}(t-t_0)^j
+\frac{1}{k!}
\int_{t_0}^t
(t-s)^k
\Big[f(\Phi_s(p))\Big]^{k+1}
v_p(s)\mathrm{d}s.\Big)$$ 
 \item  The function $\lambda$, for a fixed non-negative integer $k$ such that $\mathcal{L}_X^{k+1}R=f\mathcal{L}_X^kR$ along an orbit $\gamma_p$ where $f$ is a continuous function is given by: \begin{align*}
        \lambda(\Phi_t(p))
&=\Big(\underbrace{\int_{t_0}^t \int_{t_0}^{s_1} \dots \int_{t_0}^{s_{k-1}}}_{k \text{ integrals}} \left[ C_0 \exp\left( \int_{t_0}^{s_{k}} \int_{t_0}^{s_{k}} f(\Phi_s(p))\mathrm{ds} \right) \right] \mathrm{d}s_{k} \dots \mathrm{d}s_1 + \sum_{j=0}^{k-1} C_j \frac{(t-t_0)^j}{j!}\Big)\times\\
&\exp\left(
-\int_{t_0}^{t}
[6\varphi(\Phi_s(p))+2a]\mathrm{d} s
\right)
 \end{align*} where $C_0\in\mathbb{R}$ and $C_j = v_p^{(j)}(t_0)$ for $j = 0, 1, \dots, k-1$ represent the initial values of the successive derivatives.
  \end{enumerate}
\end{thm}
\begin{proof}
    \begin{enumerate}
        \item The fact that $\mathcal{L}_X^{j+1} R = f \, \mathcal{L}_X^j R$ for every natural number $j\in\{1,..,k\}$ is equivalent to $D_X^{j+1}(\lambda)=fD_X^j(\lambda)$. Consequently, along an orbit $\gamma(t)$, we obtain: $$e^{-\beta_p(t)}
\frac{\mathrm{d}^{j+1}}{\mathrm{d}t^{j+1}}
e^{\beta_p(t)}=f(\Phi_t(p))e^{-\beta_p(t)}
\frac{\mathrm{d}^j}{\mathrm{d}t^j}
e^{\beta_p(t)}$$ so for any function $u$ of the variable $t$, we have: $$
\frac{\mathrm{d}^{j+1}}{\mathrm{d}t^{j+1}}
\Big(e^{\beta_p(t)}u(t)\Big)=f(\Phi_t(p))\frac{\mathrm{d}^j}{\mathrm{d}t^j}\Big(e^{\beta_p(t)}u(t)\Big)$$. If we set $v_p(t)=e^{\beta_p(t)}u(t)$, then we again obtain\begin{equation}\label{z}
    v^{(j+1)}_p(t)=f(\Phi_t(p))v^{(j)}_p(t)
\end{equation}
This means that $$v'_p(t)=f(\Phi_t(p))v_p(t),\quad v''_p(t)=f(\Phi_t(p))v'_p(t)=\big[f(\Phi_t(p))\Big]^2v_p(t),.....,v^{(k+1)}_p(t)=\Big[f(\Phi_t(p))\Big]^{k+1}v_p(t).$$ By setting the initial conditions $$v_p(t_0)=c_0,\qquad
v_p'(t_0)=c_1,\qquad
\ldots,\qquad
v_p^{(k)}(t_0)=c_k,$$
The solution is the unique solution to the Cauchy problem.
$$
\begin{cases}
v_p^{(k+1)}(t)
=
\big[f(\Phi_t(p))\big]^{k+1}v_p(t),
\\
v_p^{(j)}(t_0)=c_j,
\qquad 0\leq j\leq k.
\end{cases}
$$

 We therefore obtain: $$v_p(t)=\displaystyle\sum_{j=0}^{k}\frac{c_j}{j!}(t-t_0)^j+\frac{1}{k!}\int_{t_0}^t
(t-s)^k\big[f(\Phi_s(p))\big]^{k+1}
v_p(s)\mathrm{d}s.$$
\item If we fix $k$, equation \eqref{z} becomes $v^{(k+1)}_p(t)=f(\Phi_t(p))v^{(k)}_p(t)$.

By setting $y(t) = v_p(t)$ and $\mathfrak{a}(t) = f(\Phi_t(p))$, the equation takes the form: $$y^{(k+1)}(t) = \mathfrak{a}(t) y^{(k)}(t).$$
Let $z(t) = y^{(k)}(t)$; the equation reduces to a first-order linear equation: $z'(t) = \mathfrak{a}(t) z(t)$.

The solution to this equation is obtained immediately:
$$z(t) = C_0 \exp\left( \int_{t_0}^t \mathfrak{a}(s) \, ds \right)$$
where $C_0 = z(t_0) = v_p^{(k)}(t_0)$ is a constant of integration.

Since $y^{(k)}(t) = z(t)$, the general solution $v_p(t)$ is obtained through $k$ successive integrations, combined with the integration polynomial associated with the initial conditions:
$$v_p(t) = \underbrace{\int_{t_0}^t \int_{t_0}^{s_1} \dots \int_{t_0}^{s_{k-1}}}_{k \text{ integrals}} \left[ C_0 \exp\left( \int_{t_0}^{s_{k}} f(\Phi_s(p))\mathrm{ds}\right) \right] \mathrm{d}s_{k} \dots \mathrm{d}s_1 + \sum_{j=0}^{k-1} C_j \frac{(t-t_0)^j}{j!}$$
where $C_j = v_p^{(j)}(t_0)$ for $j = 0, 1, \dots, k-1$ represent the initial values of the successive derivatives.
    \end{enumerate}
\end{proof}
\section{Conclusion}

In this work, we conducted an in-depth study of the geometry of Riemannian manifolds satisfying the rank-one anisotropic curvature condition $R = \lambda (\xi^\flat\otimes\xi^\flat)\owedge g$. After establishing that in dimension $n\ge 3$, such a manifold admits a Ricci almost-soliton structure whose characteristic vector field is not conformal, we showed that in dimension $2$, this field satisfies the equation $\mathcal{L}_X g = 2(\mu - \lambda\|\xi\|)g$. Furthermore, we proved that when the vector field of a Ricci soliton constitutes a symmetry of order $k$ for the Ricci tensor ($\mathcal{L}_X^k\operatorname{Ric}=0$), the geometric study reduces to solving a system of PDEs along the flow trajectories.

Moreover, we demonstrated that when $X$ is a conformal vector field whose Lie bracket with $\xi$ is proportional to $\xi$, the geometric problem reduces to a scalar differential problem. More precisely, for any fixed natural number $k$, there exists an operator $D = X + \alpha$ such that $\mathcal{L}_X^k R = D^k(\lambda)\eta\owedge g$. This formulation allows for an explicit analysis of the behavior of the function $\lambda$ along the integral curves of $X$, particularly under the following invariance and stationarity conditions:
\begin{itemize}[label=$\bullet$]
    \item $\mathcal{L}_X^k R = R$ ;
    \item for a continuous function $f$ on $M$, \begin{itemize}[label=$\star$]
         \item $\mathcal{L}_X^{j+1} R =f \mathcal{L}_X^j R$ for all $j \in \{0, \dots, k\}$ ;
    \item $\mathcal{L}_X^{k+1} R = f\mathcal{L}_X^k R$. 
    \end{itemize}
\end{itemize}

\end{document}